\documentclass[12pt, final]{article}

\RequirePackage{amsfonts}
\RequirePackage{amsmath}
\RequirePackage{alltt, moreverb}
\RequirePackage{showkeys, showidx}
\RequirePackage{graphicx}
\RequirePackage[utf8]{inputenc}
\RequirePackage{algorithm, algorithmic}

\RequirePackage{color, colortbl}
\RequirePackage{rotating}
\RequirePackage{boxedminipage}
\RequirePackage{upgreek}
\RequirePackage{subfigure}
\usepackage{textcomp}

\RequirePackage{enumitem}

\usepackage{amsopn}
\RequirePackage{amsthm}
\RequirePackage{amsbsy}
\theoremstyle{plain} 
\newtheorem{theorem}{Theorem}[section]

\newtheorem{proposition}[theorem]{Proposition}

\DeclareMathOperator*{\argmin}{arg\,min}                   
\renewcommand{\t} {^{\top}}                                

\newcommand{\bmx}{\left[ \begin{array}}
\newcommand{\emx}{\end{array} \right]}

\newcommand{\bfPsi}{{\boldsymbol{\Psi}}}

\newcommand{\bfA}{{\bf A}}

\newcommand{\bfH}{{\bf H}}
\newcommand{\bfI}{{\bf I}}

\newcommand{\bfL}{{\bf L}}
\newcommand{\bfM}{{\bf M}}

\newcommand{\bfQ}{{\bf Q}}
\newcommand{\bfR}{{\bf R}}

\newcommand{\bfT}{{\bf T}}
\newcommand{\bfU}{{\bf U}}
\newcommand{\bfV}{{\bf V}}

\newcommand{\bfZ}{{\bf Z}}

\newcommand{\bfb}{{\bf b}}

\newcommand{\bfd}{{\bf d}}
\newcommand{\bfe}{{\bf e}}

\newcommand{\bfr}{{\bf r}}
\newcommand{\bfs}{{\bf s}}

\newcommand{\bfu}{{\bf u}}
\newcommand{\bfv}{{\bf v}}
\newcommand{\bfw}{{\bf w}}
\newcommand{\bfx}{{\bf x}}
\newcommand{\bfy}{{\bf y}}
\newcommand{\bfz}{{\bf z}}
\newcommand{\bfzero}{{\bf0}}

\newcommand{\calR}{\mathcal{R}}

\newcommand{\bbR}{\mathbb{R}}

\newcommand{\true}{\mathrm{true}}

\usepackage{tikz,pgflibraryplotmarks}
\usepackage{pgf,pgfplots,pgfarrows} 

\newdimen\iwidth
\newdimen\iheight
\usepackage{multirow}

\usepackage{amsmath,amssymb} 
\usepackage{algorithm,algorithmic}
\newcommand{\norm}[2][]{\left\Vert#2\right\Vert_{#1}}

\newcommand{\diag}{\mathrm{diag}}

\usepackage{color}

\usepackage{fullpage}
\usepackage{marvosym}

\newcommand{\mt} {^{-\top}}                                
\newcommand{\R}{\mathbb{R}}

\newcommand{\TheTitle}{Flexible Krylov methods for {$\ell_\MakeLowercase{p}$} regularization}

\title{{\TheTitle}}

\author{
  Julianne Chung
  \thanks{Department of Mathematics, Computational Modeling and Data Analytics Division, Academy of Integrated Science, Virginia Tech, Blacksburg, VA, USA
\newline \hspace*{10ex}
    \Letter \ \texttt{jmchung@vt.edu} \ \ \Mundus \ {www.math.vt.edu/people/jmchung/}}
  \and
  Silvia Gazzola
  \thanks{Department of Mathematical Sciences, University of Bath, United Kingdon
\newline \hspace*{10ex}
    \Letter \ \texttt{{S.Gazzola@bath.ac.uk} \ \ \Mundus \ {http://people.bath.ac.uk/sg968/}}}
}

\begin{document}
\maketitle
\begin{abstract}
In this paper we develop flexible Krylov methods for efficiently computing regularized solutions to large-scale linear inverse problems with an $\ell_2$ fit-to-data term and an $\ell_p$ penalization term, for $p\geq 1$.  First we approximate the $p$-norm penalization term as a sequence of $2$-norm penalization terms using adaptive regularization matrices in an iterative reweighted norm fashion, and then we exploit flexible preconditioning techniques to efficiently incorporate the weight updates. To handle general (non-square) $\ell_p$-regularized least-squares problems, we introduce a flexible Golub-Kahan approach and exploit it within a Krylov-Tikhonov hybrid framework. The key benefits of our approach compared to existing optimization methods for $\ell_p$ regularization are that efficient projection methods replace inner-outer schemes and that expensive regularization parameter selection techniques can be avoided. Theoretical insights are provided, and numerical results from image deblurring and tomographic reconstruction illustrate the benefits of this approach, compared to well-established methods.
Furthermore, we show that our approach for $p=1$ can be used to efficiently compute solutions that are sparse with respect to some transformations.
\end{abstract}


\section{Introduction} \label{sec:introduction}
Inverse problems are prevalent in many important applications, ranging from biomedical to geophysical imaging, and solutions must be computed reliably and efficiently. In this work we are interested in linear inverse problems of the form
\begin{equation}\label{linsys}
 \bfb = \bfA \bfx_\true + \bfe\,,
\end{equation}
where $\bfb \in \bbR^m$ is the observed data, $\bfA \in \bbR^{m \times n}$ models the forward process, $\bfx_\true \in \bbR^n$ is the desired solution, and $\bfe\in \bbR^m$ represents noise or errors in the observation.  Due to the ill-posedness of the underlying problem \cite{hansen2010discrete},
regularization should be applied to recover a meaningful approximation of $\bfx_\true$ in (\ref{linsys}).
In this paper, we are interested in problems of the form
\begin{equation}
	\label{eq:pnormAb}
 \min_\bfx \norm[2]{\bfA \bfx - \bfb}^2 + \lambda \norm[p]{\bfPsi\bfx}^p\,,
\end{equation}
where $\norm[p]{\cdot}$ for $ p\geq 1$ is the vectorial $p$-norm, $\lambda >0$ is a regularization parameter, and $\bfPsi\in\bbR^{n\times n}$ is a nonsingular matrix.
For $p=2$ and $\bfPsi = \bfI$,~\eqref{eq:pnormAb} is the standard Tikhonov regularization problem, and many efficient techniques, including hybrid iterative methods, have been proposed, see,  e.g.,~\cite{chung2015hybrid,gazzola2015survey,o1981bidiagonalization,kilmer2001choosing}.  However, optimization problems (\ref{eq:pnormAb}) for $p\neq 2$ can be significantly more challenging.  For example, for $p=1$, the so-called $\ell_1$-regularized problem suffers from non-differentiability at the origin; moreover, in some situations, one may wish to consider $0<p<1$, which results in a nonconvex optimization problem, see, e.g., \cite{huang2017majorization,lanza2015generalized,lanza2017nonconvex}.  In this paper, we will focus on $p\geq1$, and henceforth we will refer to problem (\ref{eq:pnormAb}) with $\bfPsi=\bfI$ as an ``$\ell_p$-regularized problem'' and problem (\ref{eq:pnormAb}) with $\bfPsi\neq \bfI$ will be dubbed the ``transformed $\ell_p$-regularized'' problem.

Typically the transformed $\ell_p$-regularized problem arises in cases where sparsity in some frequency domain (e.g., in a wavelet domain) is desired. Depending on the application, a sparsity transform may be included in both the fit-to-data and the regularization term.  This was considered in \cite{belge2000wavelet} for image deblurring problems, where the resulting minimization problem was solved with an inner-outer iteration scheme.

Most of the previously developed methods for $\ell_p$ minimization utilize nonlinear optimization schemes or iteratively reweighted optimization schemes, which can get very expensive due to inner-outer iterations \cite{arridge2014iterated,renaut2017hybrid,wohlberg2008lp}.  Other popular approaches such as the split Bregman method \cite{goldstein2009split}, separable approximations \cite{wright2009sparse}, and accelerations of the iterative shrinkage thresholding algorithms \cite{beck2009fast}, are fast alternatives, but a main disadvantage is that the regularization parameter must be selected a priori and can be a cumbersome task. Krylov methods, on the other hand, have nice convergence and regularizing properties, so there have been recent efforts to exploit Krylov methods to solve the $\ell_p$-regularized problem.  For example, \cite{huang2017majorization,lanza2015generalized} considered generalized Krylov methods for $\ell_p-\ell_q$ minimization, and Krylov methods based on the flexible Arnoldi algorithm were considered in \cite{gazzola2014generalized,saibaba2013flexible,simoncini2002flexible}.
Our proposed methods are mostly related to the latter, which
computes approximate solutions to the $\ell_p$-regularized problem when $\bfA$ is square.  Below we outline the main distinctions and contributions of our work.

In this paper, we propose new iterative hybrid methods based on a flexible Golub-Kahan decomposition to solve $\ell_p$-regularized problems~\eqref{eq:pnormAb}, where flexible preconditioning techniques are used to build appropriate approximation subspaces for the solution.  In particular, we describe two methods, namely flexible LSQR and flexible LSMR, and show how Tikhonov regularization can be used to solve the projected problem, where the properties of the matrices associated to the flexible Golub-Kahan decomposition are exploited for efficient regularization parameter selection (in a hybrid fashion).  We underline that methods based on the flexible Golub-Kahan algorithm are matrix-free, i.e., they only require accessing $\bfA$ and $\bfA\t$ via matrix-vector multiplication. Furthermore, we describe a way to incorporate regularization terms expressed as the $\ell_p$-norm of the transformed solution within the flexible schemes (based on both the Arnoldi and the Golub-Kahan decompositions), i.e., to deal with the transformed $\ell_p$-regularized problem, $p\geq 1$.

One of the first major contributions, compared to~\cite{gazzola2014generalized}, is that these methods can be used to solve problems with general (e.g., non-square) coefficient matrix $\bfA.$ Second, we provide theoretical results that show optimality properties for the flexible approaches and show that in exact arithmetic, flexible LSMR iterates are equivalent to flexible Arnoldi-Tikhonov iterates on the normal equations. Third, contrary to classical Krylov-Tikhonov methods \cite{gazzola2015survey}, which can handle penalization terms evaluated in the 2-norm, the new methods can approximate penalization terms evaluated in the sparsity-inducing 1-norm and can include an invertible transformation. In this way we generalize the flexible Arnoldi decomposition proposed in \cite{gazzola2014generalized}, as well as the flexible Golub-Kahan decomposition derived in this paper.
Numerical comparisons to well-established $\ell_1$ regularization methods reveal that the proposed strategies provide an easy-to-use approach for computing reconstructions with similar properties, but with two significant benefits: firstly, the regularization parameters can be selected automatically thanks to the hybrid framework; secondly, information about the current solution is incorporated into the solution process as soon as it becomes available, with potentially great computational savings with respect to methods involving inner-outer iterations.

The paper is organized as follows.  In Section~\ref{sec:generalframework} we review the ideas underlying the iteratively reweighted norm (IRN) approach for $\ell_p$ regularization and briefly review the flexible Arnoldi-Tikhonov appraoch. In Section~\ref{sec:flexible} we derive the flexible Golub-Kahan decomposition, leading to the introduction of the new flexible LSQR and flexible LSMR algorithms; hybrid approaches based on flexible LSQR and flexible LSMR are addressed, with a particular emphasis on the choice of regularization term and regularization parameter. Theoretical results are provided.  In Section~\ref{sec:sparstransf} we describe how a sparsity transform can be handled within hybrid schemes based on the flexible Arnoldi and Golub-Kahan algorithms, analyzing how the approximation subspaces for the solution are modified by incorporating reweightings and sparsity transforms. Numerical results are presented in Section~\ref{sec:numerics}, and conclusions and future work are provided in Section~\ref{sec:conclusions}.

\section{Background on iteratively reweighted and flexible methods for $\ell_p$ regularization}\label{sec:generalframework}

A typical strategy for solving the $\ell_p$-regularized inverse problem is the iteratively reweighted norm (IRN) algorithm \cite{IRNekki,wohlberg2008lp}. This approach requires solving a sequence of reweighted, penalized least-squares problems where the weights change at each iteration.  When dealing with large systems, each least-squares problem is solved by an iterative method, so that an inner-outer iteration scheme is naturally established. In the following we use the acronym IRN to indicate a wide class of algorithms that leverage (outer) reweighing together with an (inner) iterative scheme. IRN methods are also closely related to the iteratively reweighted least squares (IRLS) methods \cite[Chapter 4]{LS}. Since IRN methods can get very costly, another common approach is to use iterative shrinkage thresholding algorithms \cite{beck2009fast}, where a two-step process is used.

Many of these methods assume that a good value for the regularization parameter is available \emph{a priori}, but oftentimes this is not the case.  And although there have been some recent works on selecting regularization parameters for $\ell_1$ regularization, e.g., \cite{giryes2011projected}, these can still be quite costly for very large problems. Selecting regularization parameters for $\ell_p$-regularized problems remains a tricky, yet crucial, task.
For the special case where $p=2$, significant works on hybrid methods have enabled successful simultaneous estimation of the regularization parameter and computation of large-scale reconstructions, see, e.g., \cite{kilmer2001choosing,renaut2017hybrid}.  In these hybrid frameworks, the problem is projected onto Krylov subspaces of increasing size and the task of choosing the regularization parameter is shifted to the smaller, projected problem.  However, such approaches have not been fully investigated for general $\ell_p$-regularized problems. The flexible hybrid framework for $\ell_p$-regularized problems that we describe in Section~\ref{sec:flexible} incorporates simultaneous parameter selection and is based on the IRN reformulation.

As described in \cite{wohlberg2008lp}, the first step toward an IRN approach is to define a sequence of appropriate regularization operators to break the $\ell_p$-regularized problem into a sequence of 2-norm problems,
\begin{equation}\label{eq:2normL}
 \min_\bfx \norm[2]{\bfA \bfx - \bfb}^2 + \lambda \norm[2]{\bfL(\bfx)\bfx}^2\,,
\end{equation}
where
\begin{equation}\label{eq:weights1}
\bfL(\bfx) = \diag\left((|[\bfx]_i|^{\frac{p-2}{2}})_{i=1,\ldots,n}\right)\,.
\end{equation}
Here $[\bfx]_i$ is the $i$th entry of vector $\bfx$. We remark that, when $p<2$, care is needed when defining (\ref{eq:weights1}), because division by 0 may occur if $[\bfx]_i=0$ for some $i=1,\dots,n$. To remedy to this potential issue, small thresholds $\tau_1,\tau_2>0$ are set and the matrix in (\ref{eq:weights1}) is redefined as
\begin{equation}\label{eq:weights2}
\bfL(\bfx) = \diag((f_{\tau}([|\bfx|]_i)^{\frac{p-2}{2}})_{i=1,\ldots,n})\,,\;\mbox{where}\;
f_{\tau}([|\bfx|]_i)=\begin{cases}
[|\bfx|]_i & \mbox{if $[|\bfx|]_i\geq\tau_1$}\\
\tau_2 & \mbox{if $[|\bfx|]_i<\tau_1$}
\end{cases}.
\end{equation}
Note that taking $\tau_2<\tau_1$ enforces some additional sparsity in $f_{\tau}([|\bfx|]_i)$. In the case \linebreak[4]$p=1$, the IRN approach obviously reduces the $\ell_1$-regularized problem~\eqref{eq:pnormAb} to a sequence of least-squares problems involving a weighted $\ell_2$ norm.  That is,
\begin{equation}\label{eq:norm1approx}
 \norm[1]{\bfx} \approx \norm[2]{\bfL(\bfx) \bfx}^2\,,
\end{equation}
where $\bfL(\bfx) = \diag(1/\sqrt{f_\tau(|\bfx|}))$, $f_\tau(\cdot)$ is defined as in (\ref{eq:weights2}), and the square root and absolute value operations are applied component-wise.
We remark that problem (\ref{eq:2normL}) can be equivalently reformulated as
\begin{equation}\label{eq:2normALb}
 \min_{\widehat\bfx} \norm[2]{\bfA \bfL(\bfx)^{-1} \widehat\bfx - \bfb}^2 + \lambda \norm[2]{\widehat\bfx}^2\,,
\end{equation}
where $\widehat \bfx = \bfL(\bfx) \bfx$. This transformation into standard form is computationally very convenient, as it only amounts to the inversion of a diagonal matrix.

Since considering directly (\ref{eq:2normL}) or (\ref{eq:2normALb}) is not possible in real problems where the true $\bfx$ is not available, and since we want to avoid nonlinearities, we follow the common practice of approximating the matrix $\bfL(\bfx)$ by the matrix
$\bfL_k = \bfL(\bfx_k)$, where $\bfx_k$ is an approximation of the solution at the $k$th iteration. The IRN method proposed in \cite{wohlberg2008lp} prescribes to apply, at the $k$th outer iteration, the conjugate gradient (CG) method to solve the normal equations associated to (\ref{eq:2normL}), i.e.,
\begin{equation}\label{eq:2normL_NE}
(\bfA\t\bfA + \lambda\bfL_k\t\bfL_k)\bfx = \bfA\t\bfb\,,\quad \bfL_k=\bfL(\bfx_k)\,.
\end{equation}
Also preconditioned CG (PCG) can be applied at the $k$th outer iteration of IRN to solve the normal equations associated to (\ref{eq:2normALb}), i.e.,
\begin{equation}\label{eq:2normALb_NE}
(\bfL_k\mt\bfA\t\bfA\bfL_k^{-1} + \lambda\bfI)\widehat\bfx = \bfL_k\mt\bfA\t\bfb\,,\quad \bfL_k^{-1}\widehat\bfx=\bfx,
\quad\bfL_k=\bfL(\bfx_k)\,.
\end{equation}
We refer to this approach as preconditioned IRN (PIRN) method, which is similar in essence to the inner-outer scheme proposed in \cite{arridge2014iterated} to handle total variation regularization. In both equations (\ref{eq:2normL_NE}) and (\ref{eq:2normALb_NE}), $\bfx_k$ is the approximation of the solution obtained at the $(k-1)$st outer iteration. We emphasize that the term ``preconditioned'' is used in a somewhat nonconventional way: the ``preconditioners'' considered here are not aimed at accelerating the convergence of the iterative solvers, but rather at enforcing some specific regularity into the associated solution subspace. Transformed $\ell_p$-regularized problems can be suitably expressed in this framework too, as we will explain in Section~\ref{sec:sparstransf}.
We stress once more that, in the IRN framework, the matrix $\bfL = \bfL_k$
changes at each outer iteration, resulting in a sequence of least-squares problems to be solved.
A more efficient alternative that is applied directly to problem (\ref{eq:2normALb}) and that exploits flexible preconditioning to bypass inner-outer iterative schemes is summarized below.

\paragraph{Generalized Arnoldi-Tikhonov approaches}
For completeness, we provide a brief overview of the generalized Arnoldi-Tikhonov (GAT) \cite{gazzola2014generalized} approach to solve problem (\ref{eq:2normALb}) for $\bfA \in \bbR^{n \times n}$ and for changing preconditioners $\bfL_k$.
Consider the flexible preconditioned Arnoldi algorithm where, at the $k$th iteration, we have
\begin{equation}\label{eq:flexiArn}
\bfA \widehat\bfZ_k = \widehat\bfV_{k+1} \widehat\bfH_k
\end{equation}
where $\widehat\bfH_k \in \bbR^{(k+1) \times k}$ is upper Hessenberg, $\widehat\bfV_k = \begin{bmatrix}
 \widehat \bfv_1 & \ldots & \widehat \bfv_k
\end{bmatrix}$ contains orthonormal columns, and $\widehat \bfZ_k = \begin{bmatrix}
 \bfL_1^{-1} \widehat \bfv_1 & \ldots & \bfL_k^{-1} \widehat \bfv_k
\end{bmatrix} \in \bbR^{n \times k}$. If we are given an initial guess $\bfx_0$ for the solution, then $\widehat \bfv_1 = \bfr_0/\norm[2]{\bfr_0}$.  We also note that, if the preconditioner is fixed along the iterations ($\bfL_i=\bfL$, $i=1,\dots,k$), then $\widehat \bfZ_k=\bfL^{-1}\widehat \bfV_k$, i.e., decomposition (\ref{eq:flexiArn}) reduces to the one associated to the standard right-preconditioned GMRES. The GAT method computes approximate solutions of the form $\bfx_k = \bfx_0+\widehat \bfZ_k \widehat \bfy_k$ where
\begin{equation}
  \label{eq:GAT}
\widehat \bfy_k = \argmin_{\bfy} \norm[2]{\widehat \bfH_k \bfy - \norm[2]{\bfr_0} \bfe_1}^2 + \lambda \norm[2]{\bfy}^2.
\end{equation}
For $\lambda = 0,$ we have the flexible GMRES (FGMRES) method \cite[Chapter 9]{saad2003iterative}. The main advantages of this approach are that \emph{only one} solution subspace needs to be generated (versus multiple solves in IRN), one matrix-vector multiplication with $\bfA$ is required at each iteration (versus one with $\bfA$ and one with $\bfA\t$ in CGLS), and the regularization parameter and stopping iteration can be computed automatically by exploiting the hybrid framework. In \cite{gazzola2014generalized}, the GAT method and its variants were used to efficiently compute approximate solutions to $\ell_1$-regularized problems, but a limitation is that this method only works for square problems.  A na\"{i}ve extension of the GAT method to general least-squares problems by applying the flexible Arnoldi algorithm to the normal equations is not recommended, due to known complications of forming and working with the normal equations \cite{golub2012matrix}.  In the following section, we exploit some new computational tools from numerical linear algebra, namely the flexible Golub-Kahan method, so that we can work directly with the residual from the original least-squares problem (\ref{eq:2normALb_NE}).

\section{Flexible Golub-Kahan hybrid methods}
\label{sec:flexible}
In this section, we describe flexible hybrid approaches based on the flexible Golub-Kahan process for solving the variable-preconditioned Tikhonov problem,
\begin{equation}
 \min_\bfx \norm[2]{\bfA \bfx - \bfb}^2 + \lambda \norm[2]{\bfL_k \bfx}^2\,,
\end{equation}
where $\bfL_k$ may change at each iteration.
Similarly to the GAT method, the flexible Golub-Kahan hybrid methods generate a basis for the solution (which takes into account a changing preconditioner in a flexible framework) and compute an approximate solution to the inverse problem by solving an optimization problem in the projected subspace (where regularization can be done efficiently and with automatic regularization parameter selection for the projected problem).
These iterative approaches are ideal for problems where $\bfA$ and $\bfA\t$ can be accessed only by matrix-vector multiplication, where only a few basis vectors are required to obtain a good solution, and where the regularization parameter is not known a priori.

\subsection{Incorporating weights: a flexible Golub-Kahan decomposition}
\label{sub:FGK}
To be able to incorporate a changing preconditioner, we use a flexible variant of the Golub-Kahan bidiagonalization (GKB) to generate a basis for the solution.  We call this the \emph{flexible Golub-Kahan} (FGK) process and mention that it is closely related to the inexact Lanczos process \cite{van2004inexact,simoncini2007recent}.  Given $\bfA, \bfb,$ and changing preconditioners $\bfL_k$, the FGK iterative process generates vectors $\bfz_k$, $\bfv_k$, and $\bfu_{k+1}$ at the $k$th iteration such that
\begin{equation}
	\label{eq:flexGK}
	\bfA \bfZ_k = \bfU_{k+1} \bfM_k \quad \mbox{and} \quad \bfA\t \bfU_{k+1} = \bfV_{k+1} \bfT_{k+1},
\end{equation}
where $\bfZ_k = \begin{bmatrix} \bfL_1^{-1} \bfv_1 & \cdots & \bfL_k^{-1}\bfv_k \end{bmatrix} \in \bbR^{n \times k}$, $\bfM_k \in \bbR^{(k+1)\times k}$ is upper Hessenberg, $\bfT_{k+1}\in \bbR^{(k+1)\times (k+1)}$ is upper triangular, and $\bfU_{k+1} = \begin{bmatrix}
 \bfu_1 & \ldots & \bfu_{k+1}
\end{bmatrix}\in\bbR^{m \times (k+1)}$ and $\bfV_{k+1} =\begin{bmatrix}
 \bfv_1 & \ldots & \bfv_{k+1}
\end{bmatrix}\in \bbR^{n\times (k+1)}$ contain orthonormal columns.
For simplicity, we let $\bfx_0=\bfzero$ and $\bfu_1 = \bfb/\norm[2]{\bfb}$, but extensions to include $\bfx_0\neq\bfzero$ are trivial and follow standard derivations.
Compared to the standard GKB \cite{golub1965calculating}, the key differences are that we now have an upper Hessenberg and an upper triangular matrix, instead of {one} bidiagonal matrix.  Also, we must keep track of an additional set of vectors, namely the basis vectors in $\bfZ_k$.  Furthermore, since there is no bidiagonal structure to exploit, the additional computational requirement is orthogonalization with all previous vectors.
However, as for standard GKB, the computational cost per iteration is dominated by a matrix-vector product with $\bfA$ and one with $\bfA\t$. We remark that, if $\bfL_k = \bfL$, (\ref{eq:flexGK}) reduces to the right-preconditioned GKB.  The FGK process is summarized in Algorithm~\ref{alg:flexGK}.

\begin{algorithm}[!ht]
\begin{algorithmic}[1]
\STATE Initialize $\bfu_1 = \bfb/\beta_1,$ where $\beta_1 = \norm{\bfb}$
\FOR {i=1, \dots, k}
\STATE Compute $\bfw = \bfA\t \bfu_i$, $t_{ji} = \bfw\t \bfv_j$ for $j=1, \ldots, i-1$
\STATE Set $\bfw = \bfw - \sum_{j=1}^{i-1} t_{ji} \bfv_j$, compute $t_{ii} = \norm{\bfw}$ and take $\bfv_i = \bfw/t_{ii}$
\STATE Compute $\bfz_i = \bfL_i^{-1} \bfv_i$ and $\bfw = \bfA \bfz_i$
\STATE  $m_{ji} = \bfw\t \bfu_j$ for $j=1, \ldots, i$ and set $\bfw = \bfw - \sum_{j=1}^i m_{ji} \bfv_j$
\STATE Compute $m_{i+1,i} = \norm{\bfw}$ and take $\bfu_{i+1} = \bfw/m_{i+1,i}$
\ENDFOR
\end{algorithmic}
\caption{Flexible Golub-Kahan (FGK) Process}
\label{alg:flexGK}
\end{algorithm}

Notice that the column vectors of $\bfZ_k$ no longer span a Krylov subspace, but they do provide a basis for the solution.  In Section \ref{sec:numerics} we provide some qualitative observations regarding the basis vectors.  For now, consider the data-fit term.  Using the relationships in~\eqref{eq:flexGK}, the projected residual can be written as
\begin{equation}
 \bfA \bfZ_k \bfy - \bfb  = \bfU_{k+1}(\bfM_k \bfy - \beta_1\bfe_1)  \label{eq:projresidual}
\end{equation}
where $\bfe_1 \in\bbR^{k+1}$ is the first column of the identity matrix of order $k+1$. Analogous to the mathematical definitions of LSQR and LSMR iterates in \cite{paige1982lsqr,paige1982algorithm,fong2011lsmr}, we define flexible LSQR (FLSQR) and flexible LSMR (FLSMR) iterates as $\bfx_k = \bfZ_k \bfy_k$, where
\begin{equation}
  \label{eq:FLSQR}
  \bfy_k = \argmin_\bfy \norm[2]{\bfM_k \bfy - \beta_1\bfe_1}^2
  \end{equation}
  and
  \begin{equation}
    \label{eq:FLSMR}
    \bfy_k = \argmin_\bfy \norm[2]{\bfT_{k+1} \bfM_k \bfy- \beta_1 t_{11} \bfe_1}^2,
    \end{equation}
  respectively.  The {FLSMR} formulation exploits the following relationships
  \begin{equation*}
    \bfA\t (\bfA \bfZ_k \bfx - \bfb) = \bfV_{k+1} (\bfT_{k+1} \bfM_k \bfy- t_{11} {\beta_1} \bfe_1)\quad \mbox{and} \quad \bfA\t \bfb  = \bfV_{k+1} t_{11} {\beta_1} \bfe_1\,.
  \end{equation*}

We have the following optimality properties for FLSQR and FLSMR that are similar to the FGMRES~\cite{saad2003iterative} ones, and that are analogous to the ones enjoyed by the standard counterparts of these methods.
\begin{proposition}
  \label{proposition}
  The FLSQR solution $\bfx_k$ obtained at the $k$th step minimizes the residual norm $\norm[2]{\bfA \bfx_k -\bfb}$ over $\bfx_0+{\rm span}\{\bfZ_k\},$
  and the FLSMR solution $\bfx_k$ obtained at the $k$th step minimizes $\norm[2]{\bfA\t(\bfA \bfx_k -\bfb)}$ over $\bfx_0+{\rm span}\{\bfZ_k\}.$
\end{proposition}

We note that FLSQR is mathematically equivalent to the full-recurrence flexible conjugate gradient method \cite{golub1999inexact,notay2000flexible} applied to the normal equations, but the advantages of this formulation are that we avoid working directly with the normal equations, and there is a natural mean to evaluate residuals for the original system. In this respect, FLSQR is comparable to the FCGLS method in \cite{NNFCGLS}.

 Furthermore, we note that $\bfT_{k+1} \bfM_k$ is a $(k+1) \times k$ upper Hessenberg matrix and that the solution subspace generated by the FGK process is the same as the one generated by the flexible Arnoldi algorithm applied to the normal equations with initial guess $\bfx_0 =\bf0$. More precisely, the following equivalence theorem holds.
\begin{theorem}
 \label{thm:equiv}
 Let $\bfA \in \bbR^{m \times n}, m \geq n$ (full column rank), $\bfb \in \bbR^m,$ $\bfx_0 = \bf0,$ and 
 take the preconditioners $\bfL_i, i = 1, 2, \ldots k$. Then, in exact arithmetic, the $k$th iterate of FLSMR is equivalent to the $k$th iterate of FGMRES applied to the normal equations
 \begin{equation}
  \bfA\t\bfA \bfx = \bfA\t\bfb.
 \end{equation}
\end{theorem}
\begin{proof}
Note that, after $k$ iterations of FGMRES applied to the normal equations, we have upper Hessenberg matrix $\widehat\bfH_k \in \bbR^{(k+1) \times k}$, matrix $\widehat\bfV_{k+1} \in \bbR^{n\times (k+1)} $ with orthonormal columns and matrix $\widehat \bfZ_k = \begin{bmatrix} \bfL_1^{-1}\widehat\bfv_1 & \ldots & \bfL_k^{-1}\widehat\bfv_k \end{bmatrix} \in \bbR^{n \times k}$ that satisfies the relationship
\begin{equation}
 \bfA\t\bfA \widehat \bfZ_k = \widehat\bfV_{k+1} \widehat\bfH_k\,.
\end{equation}
The projected problem is given by
\begin{equation}
  \label{eqn:GMRESproj}
 \min_{\bfx \in \calR(\widehat\bfZ_k)} \norm[2]{\bfA\t \bfA \bfx - \bfA\t \bfb}^2
  = \min_{\bfy}\norm[2]{\widehat \bfH_k  \bfy - \norm[2]{\bfA\t\bfb}\bfe_1}^2\,,
\end{equation}
so the $k$th iterate of FGMRES (assuming no breakdown) is given by \begin{equation*}
\widehat\bfx_k = \widehat \bfZ_k \widehat \bfH_k^\dagger \norm[2]{\bfA\t\bfb}\bfe_1.
\end{equation*}
In exact arithmetic the solution subspaces generated by FGMRES and FGK in Algorithm~\ref{alg:flexGK} are the same, and coincide with
\begin{equation*}
 {\rm span}\{\bfL_1^{-1} \bfA\t \bfb, (\bfL_2^{-1} \bfA\t \bfA)(\bfL_1^{-1} \bfA\t \bfb), \ldots, (\bfL_k^{-1}\bfA\t\bfA)\cdots (\bfL_2^{-1}\bfA\t\bfA)(\bfL_1^{-1}\bfA\t\bfb) \}\,,
\end{equation*}
so that $\widehat\bfZ_k = \bfZ_k$.
The optimality condition for FGMRES (see Proposition~9.2 in~\cite{saad2003iterative}) and FLSMR (see Proposition~\ref{proposition}) guarantee that the $k$th iterate of FLSMR and FGMRES both correspond to the solution of~\eqref{eqn:GMRESproj}.
\end{proof}

\subsection{Solving the regularized problem: flexible hybrid algorithms}
\label{sub:hybridFGK}
As described in Section~\ref{sub:FGK}, the FGK process can be used to build a solution subspace that can efficiently incorporate changing preconditioners, and one can solve the projected problems~\eqref{eq:FLSQR} and~\eqref{eq:FLSMR}, which correspond to the FLSQR and FLSMR methods respectively. However, it is well-known that, for inverse problems, iterative methods exhibit a semiconvergent behavior, where the reconstruction errors $\norm[2]{\bfx_k - \bfx_\true}/\norm[2]{\bfx_\true}$ decrease initially but at some point increase due to amplification of noise \cite{hansen2010discrete}.  This phenomenon, which is common for most ill-posed inverse problems, {occurs also for flexible methods, as} can be seen in Figure~\ref{fig:semiconvergence}(a).

Hybrid methods, where regularization is included on the projected problem, have been proposed as means to stabilize the reconstruction errors. The first hybrid approach that we propose is analogous to the GAT algorithm (c.f. equation~\eqref{eq:GAT}), where we include a standard regularization term in~\eqref{eq:FLSQR}, so that
\begin{equation}
	\label{eq:proj_I}
 \bfy_k = \argmin_\bfy \norm[2]{\bfM_k \bfy -  \beta \bfe_1}^2 + \lambda \norm[2]{\bfy}^2.
\end{equation}
Henceforth, we refer to $\bfx_k = \bfZ_k \bfy_k$, where $\bfy_k$ is defined in~\eqref{eq:proj_I}, as FLSQR-I iterates.

We also consider a hybrid subspace optimization method called FLSQR-R where iterates are constructed as $\bfx_k = \bfZ_k \bfy_k$ where
\begin{equation}
	\label{eq:proj_R}
 \bfy_k = \argmin_\bfy \norm[2]{\bfM_k \bfy -  \beta \bfe_1}^2 + \lambda \norm[2]{\bfR_k\bfy}^2\,.
\end{equation}
Here, $\bfR_k$ comes from the thin QR factorization $\bfZ_k = \bfQ_k \bfR_k$, which is inexpensive to compute if $k$ is not too large.
The FLSQR-R method exhibits some desirable properties, especially for inverse problems.  First, the FLSQR-R iterates can be interpreted as a best approximation in a subspace, in that $\bfx_k$ solves
\begin{equation}
  \label{eq:FLSQR-Requiv}
  \min_{\bfx \in \calR(\bfZ_k)} \norm[2]{\bfA \bfx - \bfb}^2 + \lambda \norm[2]{\bfx}^2\,.
\end{equation}
Hence, the regularization parameter $\lambda$ corresponds to regularization for the original full-dimensional inverse problem.  Second, using a reformulation of the FLSQR-R subproblem~\eqref{eq:proj_R}, we can show that the singular values of the coefficient matrix $\bfM_k \bfR_k^{-1}$ provide good approximations to the singular values of $\bfA.$  Indeed, we can see this by considering the following relations,
\begin{align}
\bfR_k^{-\top}\bfM_k\t\bfM_k \bfR_k^{-1} & = \bfR_k^{-\top}\bfM_k\t \bfU_{k+1}\t \bfU_{k+1} \bfM_k\bfR_k^{-1} \nonumber\\
& = \bfR_k^{-\top}\bfZ_k\t \bfA\t \bfA \bfZ_k\bfR_k^{-1}\label{SingValApprox}\\
& = \bfQ_k\t \bfA\t \bfA \bfQ_k,\nonumber
\end{align}
where $\bfQ_k \in \bbR^{n \times k}$ contains orthonormal columns.
Since the eigenvalues are just the squares of the singular values, we see that as $k$ increases, the singular values of $\bfM_k \bfR_k^{-1}$ provide better approximations to the singular values of $\bfA$.

Hybrid LSMR variants, namely the FLSMR-I and FLSMR-R methods, can be defined analogously.  Then, using Theorem~\ref{thm:equiv}, we can see that in exact arithmetic and for a fixed regularization parameter, FLSMR-I iterates are equivalent to GAT iterates applied to the normal equations with initial guess $\bf0$.  However, the benefit of the FGK approaches versus GAT on the normal equations is that FGK produces residual norms for the original problem, which can be important for tools such as the discrepancy principle for parameter selection and for stopping criteria.

Unless otherwise stated, the parameter choice methods considered here are based on the discrepancy principle: in particular, we either prescribe the discrepancy principle to be satisfied at each iteration, or we apply the ``secant update'' variant prescribing suitable updates of the regularization parameter at each iteration (see ~\cite{gazzola2014generalized} and \cite{gazzola2015survey} for a description of these regularization parameter selection and stopping criteria strategies, which can be trivially extended to work with flexible decompositions).

\paragraph{An Illustration} The goals of this illustration are (1) to demonstrate the higher quality of the solutions obtained by applying flexible methods (due to a better basis), (2) to motivate the need for a hybrid approach (by showing semiconvergence behavior of FLSQR and FLSMR), and (3) to show that the singular values of the original problem can be approximated well by using FLSQR-R.  More thorough numerical results and comparisons will be presented in Section~\ref{sec:numerics}.

For this illustration, we use the \texttt{heat} example from RegularizationTools \cite{hansen2010discrete}, where $\bfA$ is $512 \times 512$, having a sparse true solution (so that $\bfPsi=\bfI$ and $p=1$ in (\ref{eq:pnormAb})). White noise is added to the observed signal at a noise level of $10^{-4}$.
\begin{figure}[!bt]
 \begin{center}
   \begin{tabular}{cc}
 \includegraphics[scale=0.30]{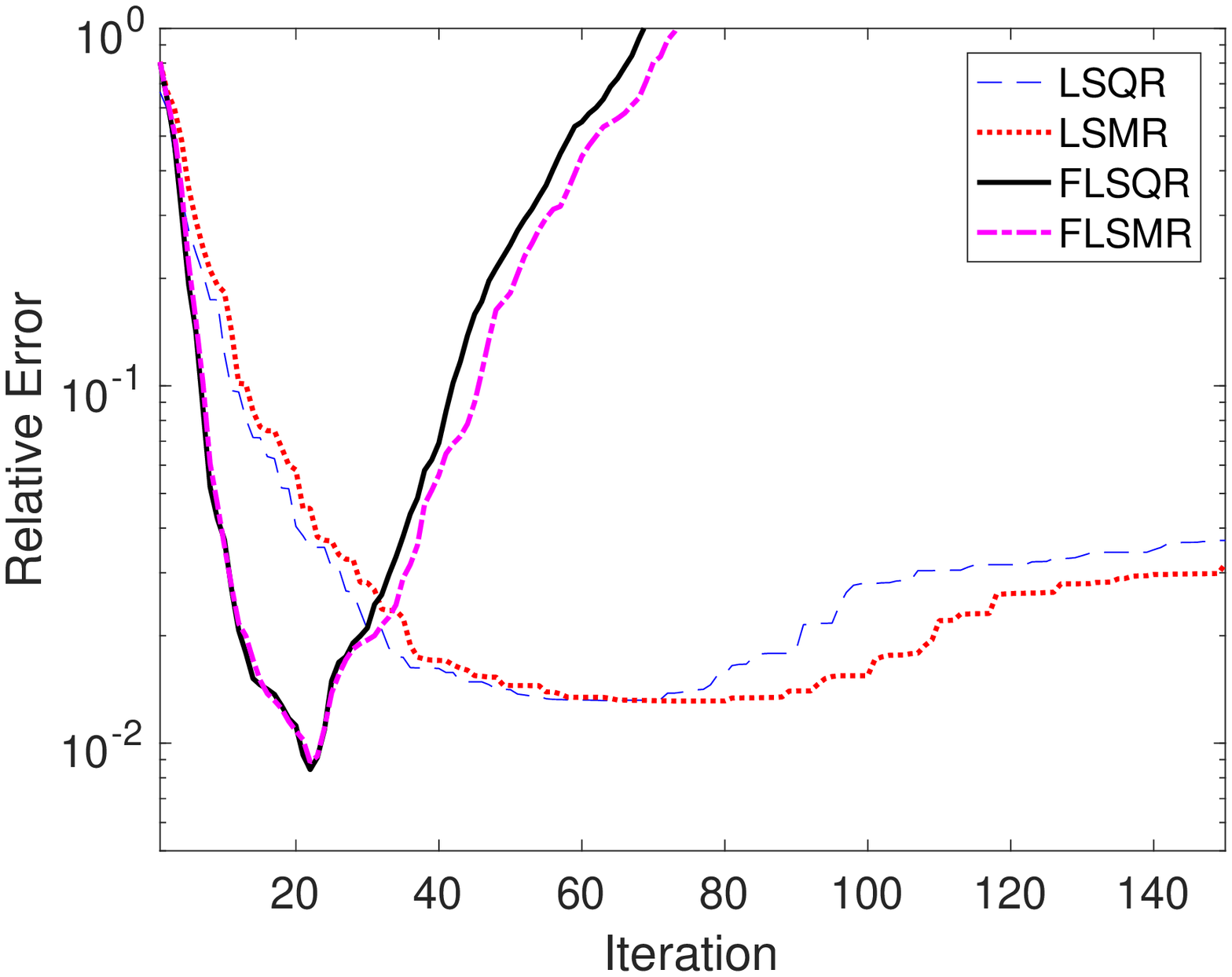} &
   \includegraphics[scale=0.30]{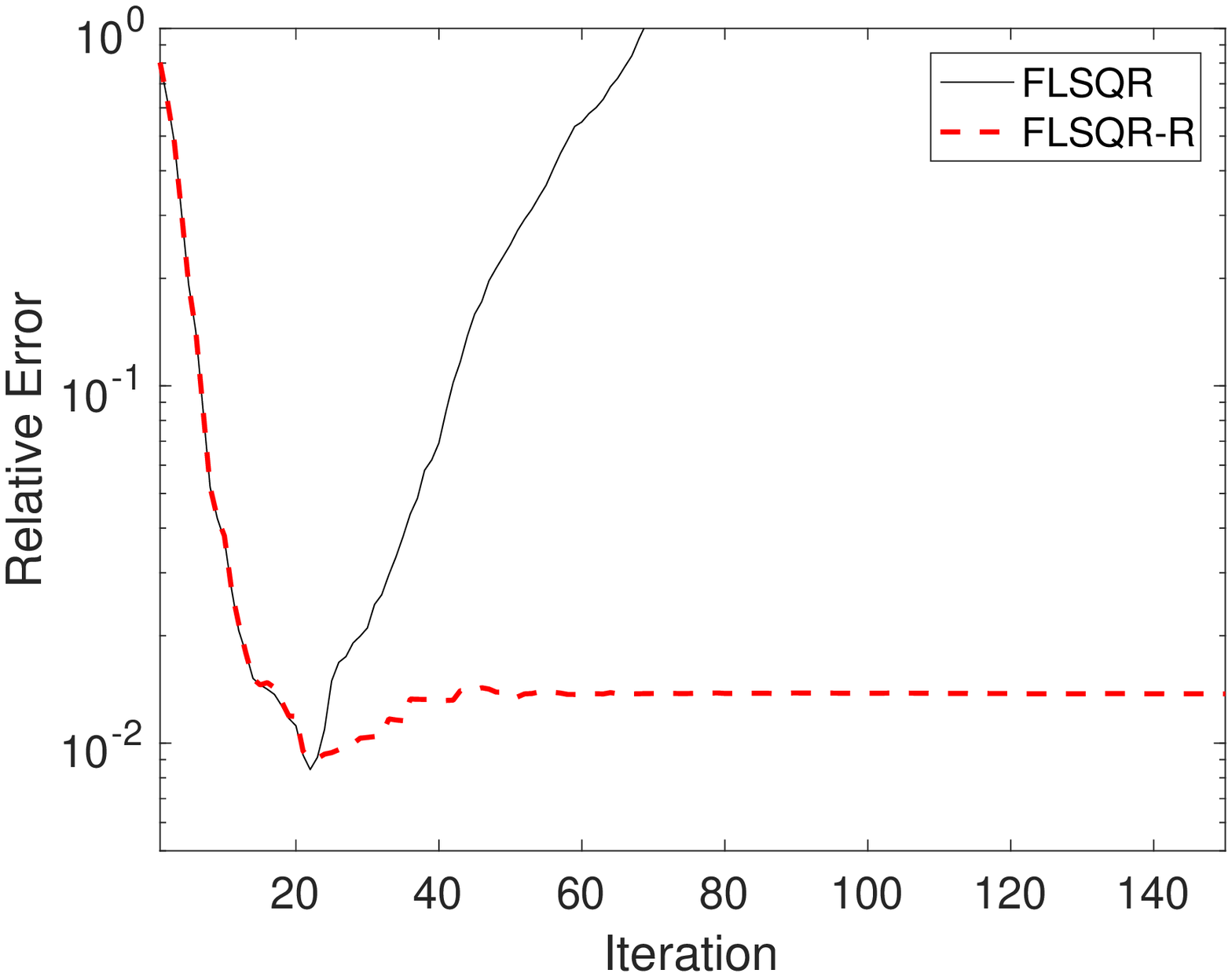}\\
   (a) & (b)
   \end{tabular}
 \end{center}
\caption{\emph{\texttt{heat}} test problem from \emph{\cite{hansen2010discrete}}. (a) relative errors, $\norm[2]{\bfx_k-\bfx_\true}/\norm[2]{\bfx_\true}$, for LSQR, LSMR, FLSQR, and FLSMR.  The semiconvergence behavior is evident.  (b) relative errors for FLSQR-I and FLSQR-R, with optimal regularization parameter.}
   \label{fig:semiconvergence}
\end{figure}
In Figure~\ref{fig:semiconvergence}(a), we provide relative reconstruction errors per iteration for LSQR, LSMR, FLSQR, and FLSMR.  The delayed semiconvergence of LSMR versus LSQR was noted in~\cite{chung2015hybrid}, and is also slightly visible for FLSMR versus FLSQR.  The more pronounced feature that we see here is that the flexible variants converge faster but also exhibit stronger semiconvergence in that the relative errors increase faster.  Thus, there is a greater need for additional regularization.  In Figure~\ref{fig:semiconvergence}(b), we show that the hybrid methods (here with the optimal regularization parameter) can stabilize the semiconvergence behavior.  Comparisons with different parameter selection methods can be found in Section~\ref{sec:numerics}.
We note that, for this particular test problem, flexible preconditioning speeds up the convergence of the iterative method. However, for our problems of interest (e.g., $\ell_p$-regularized problems), flexible preconditioning is mainly used to improve the solution subspace.  Thus, the particular choice of regularization for the projected problem is not so critical and is mostly required for stabilization of iterates.

Another important tool for the analysis of a regularization method is the approximation of the singular values of $\bfA.$  In the standard GKB, it is well known that the singular values of the bidiagonal matrix approximate the singular values of $\bfA.$  However, these results do not directly extend to the FGK process. In Figure~\ref{fig:singularvalues}, we provide the singular values of $\bfA$ in the dashed line, which is partially covered by the FLSQR-R curve.  Then, for $k=20$ to $k=420$ in intervals of $100$, we provide the singular values of upper Hessenberg matrix $\bfM_k$ for FLSQR, and the singular values of $\bfM_k \bfR_k^{-1}$ for FLSQR-R.
\begin{figure}[!b]
  \includegraphics[width=\textwidth]{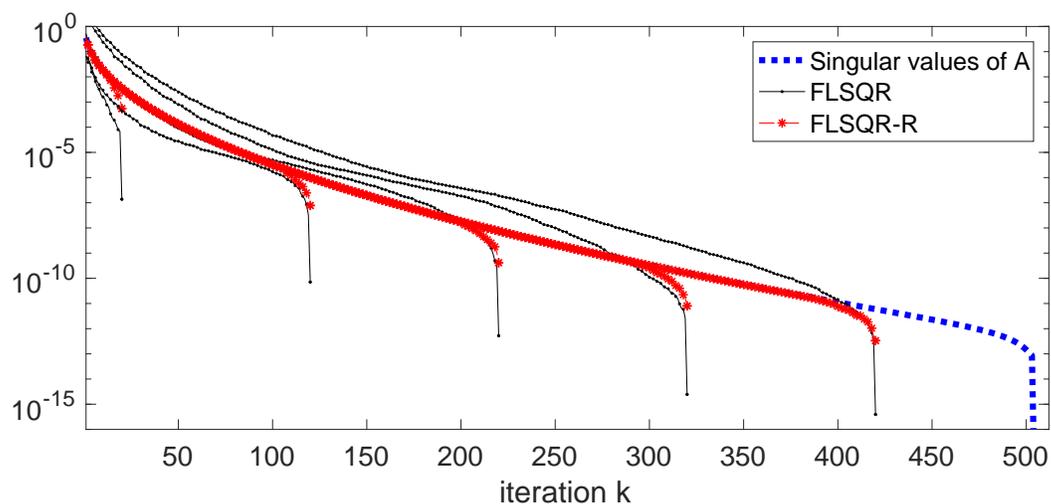}
\caption{\emph{\texttt{heat}} test problem from \emph{\cite{hansen2010discrete}}. This plot compares the singular values of $\bfA$ to the singular values of $\bfM_k$ from FLSQR and of $\bfM_k \bfR_k^{-1}$ from FLSQR-R, for iterations $k$ between $20$ and $420$ in increments of $100$.}\label{fig:singularvalues}
\end{figure}
Note that, in the flexible methods, the previous iterate $\bfx_{k-1}$, which may include regularization, changes the preconditioner and hence the FGK matrices.  It is evident that singular values of $\bfM_k \bfR_k^{-1}$ from FLSQR-R provide better approximations to the singular values of $\bfA$ than those of $\bfM_k$ from FLSQR.

\section{Flexible methods for the transformed problem}
\label{sec:sparstransf}

As mentioned in the Introduction, the goal in many applications is to compute solutions that are sparse with respect to some transformation (e.g., in some frequency domain).  In this section, we focus on flexible Arnoldi and {flexible} Golub-Kahan hybrid methods for solving the transformed $\ell_p$-regularized problem~\eqref{eq:pnormAb} where $\bfPsi \neq \bfI.$ Although any invertible transformation matrix can be used here, we will focus on wavelet transforms mainly for two reasons. Firstly, it is well known that many images can be sparsely represented in the wavelet domain. Indeed, wavelet-based iterative methods have been considered for linear inverse problems, see, e.g., \cite{daubechies2004iterative,espanol2014wavelet,klann2011wavelet,tropp2010computational}. Secondly, when taking orthonormal wavelet transforms,
computations involving $\ell_2$-norms of transformed quantities or inverse transforms can be easily performed. The specific strategy used to incorporate a wavelet transform into the flexible iterative solvers depend on the properties of the linear system at hand (which, eventually, depends on the properties of the inverse problem to be regularized) and, for all the methods, the regularization parameter can be automatically estimated.

Let $\widetilde \bfPsi\in\R^{m\times m}$ be an orthogonal matrix.  Then, it easy to see that problem~\eqref{eq:pnormAb}  is equivalent to
\begin{equation}\label{eq:transproblem}
 \min_\bfx \norm[2]{\widetilde \bfPsi \bfA \bfPsi^{-1} \bfPsi \bfx - \widetilde\bfPsi \bfb}^2 + \lambda \norm[p]{\bfPsi \bfx}^p\,.
\end{equation}
 Moreover, after some variable transformations,~\eqref{eq:transproblem} can be written as
\begin{equation}
  \label{eq:transformed}
 \min_\bfs \norm[2]{\bfH \bfs - \bfd}^2 + \lambda \norm[p]{\bfs}^p\,,\quad\mbox{where}\quad \bfH=\widetilde \bfPsi \bfA \bfPsi^{-1},\; \bfs=\bfPsi\bfx\,,\; \bfd =\widetilde\bfPsi \bfb\,,
\end{equation}
which is an $\ell_p$-regularized problem. The choice of $\widetilde \bfPsi$ is problem-dependent and solver-dependent.

For instance, when considering image deblurring problems where both $\bfx$ and $\bfb$ are images of the same size described by pixel values, it is natural to take $\widetilde \bfPsi=\bfPsi$ to be an orthogonal wavelet transform; this formulation was considered
in~\cite{belge2000wavelet}. If the GAT method is applied to solve problem (\ref{eq:transformed}) with $p=1$ and variable preconditioner $\bfL(\bfs_k) = \bfL_k$, then the following approximation subspace for the $k$th approximation of the transformed solution $\bfs$ is generated:
\begin{equation*}
 {\rm span}\{\bfL_1^{-1} \bfd,\bfL_2^{-1}\bfH\bfL_1^{-1}\bfd,\ldots,\bfL_k^{-1}\bfH\cdots \bfL_2^{-1}\bfH\bfL_1^{-1}\bfd \}\,.
\end{equation*}
This subspace enforces sparsity in the wavelet domain for the wavelet coefficients $\bfs$ of the original image $\bfx$. The approximation subspace for the latter is given by
\begin{equation}\label{tansfApprSubsp}
\bfPsi\t {\rm span}\{\bfL_1^{-1}\bfPsi\bfb,\bfL_2^{-1}\bfPsi\bfA\bfPsi\t\!\bfL_1^{-1}\bfPsi\bfb,\ldots,\bfL_k^{-1}\bfPsi\bfA\bfPsi\t\!\cdots \bfL_2^{-1}\bfPsi\bfA\bfPsi\t\!\bfL_1^{-1}\bfPsi\bfb \}\,,
\end{equation}
so that it is evident that first sparsity is enforced in the wavelet domain, and then the sparse wavelet coefficients are transformed back into the original pixel domain. However, for situations where one has no intuition regarding the sparsity properties of $\bfb$, one can simply take $\widetilde \bfPsi = \bfI$.

Analogously, if solvers based on the FGK process are applied to solve the same problem, then the approximation subspace for $\bfx$ is given by
\begin{eqnarray}
\bfPsi\t {\rm span}\left\{\bfL_1^{-1}\bfPsi\bfA\t\bfb,(\bfL_2^{-1}\bfPsi\bfA\t\bfA\bfPsi\t)\bfL_1^{-1}\bfPsi\bfA\t\bfb,\ldots,\right.\nonumber\\
\left.(\bfL_k^{-1}\bfPsi\bfA\t\bfA\bfPsi\t)\cdots (\bfL_2^{-1}\bfPsi\bfA\t\bfA\bfPsi\t)\bfL_1^{-1}\bfPsi\bfA\t\bfb \right\}\,.\label{FGKtansfsp}
\end{eqnarray}
Notice that the choice of $\widetilde \bfPsi$ is irrelevant for flexible methods based on FGK since
\[
\bfH\t\bfd=\bfPsi\bfA\t\widetilde\bfPsi\t\widetilde\bfPsi\bfb=\bfPsi\bfA\t\bfb,\quad
\bfH\t\bfH = \bfPsi\bfA\t\widetilde\bfPsi\t\widetilde\bfPsi\bfA\bfPsi\t= \bfPsi\bfA\t\bfA\bfPsi\t.
\]

\paragraph{An Illustration}
The goal of this illustration is to show that the approximation space generated by the flexible Arnoldi algorithm applied to problem (\ref{eq:transformed}) is more suitable than the one generated by its standard counterpart. We consider a 1D signal $\bfx$ with 64 entries, generated in such a way that only 8 of its 1-level Haar wavelet coefficients $\bfs$ are nonzero. The signal is corrupted by Gaussian blur with variance 2.25 and band 5, and white noise of level $10^{-2}$ is added. The exact and corrupted signals are displayed in Figure \ref{fig:BasisVectors}(a), and their wavelet coefficients are displayed in Figure \ref{fig:BasisVectors}(b). We choose $\lambda=0$ in (\ref{eq:transformed}) so that the approximation subspace (\ref{tansfApprSubsp}) does not depend on the specific parameter choice strategy that one may wish to consider. The threshold $\tau_1$ in (\ref{eq:weights2}) is set to 0.2, while $\tau_2$ is close to machine precision. Figure \ref{fig:BasisVectors}(c) displays the best reconstructions obtained by the FGMRES (11th iteration) and the GMRES (30th iteration) methods. One can clearly see that the FGMRES solution is of much higher quality than the GMRES one, and that the wavelet coefficients of the FGMRES solution are much sparser than the GMRES ones (see Figure \ref{fig:BasisVectors}(d)). The good performance of FGMRES for this example can be explained {by} looking at some of the basis vectors for the approximation of the solution, displayed in Figure \ref{fig:BasisVectors}(e)--(h). Indeed, the preconditioned basis vectors for the signal $\bfx$ have a quite piecewise-constant behavior, while the unpreconditioned ones display spurious oscillations; correspondingly, the preconditioned basis vectors for the wavelet coefficients $\bfs$ have a clear sparsity pattern, which is not reproduced by the unpreconditioned ones. Therefore, the FGMRES solution is better than the GMRES one as it is obtained by combining better basis vectors for the approximation subspace. We remark that the basis vectors generated from the FGK process have similar properties, and thus are omitted. Also, a similar behavior of the preconditioned basis vectors can be observed in the more challenging experiments presented in Section \ref{sec:numerics}.

\begin{figure}[tbp]
\centering
\begin{tabular}{cc}
\hspace{-0.7cm}\includegraphics[width=6.3cm]{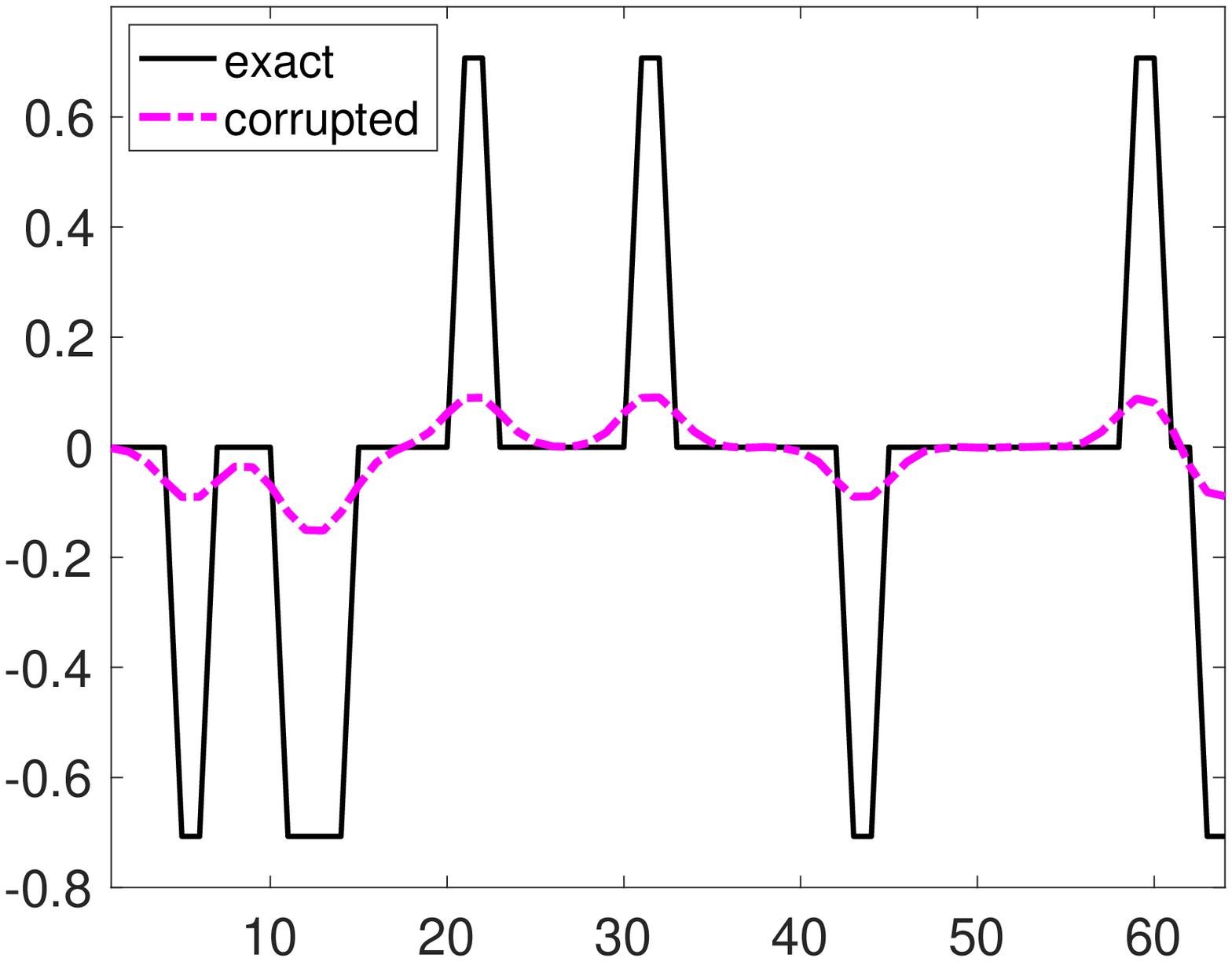} &
\hspace{-0.7cm}\includegraphics[width=6.3cm]{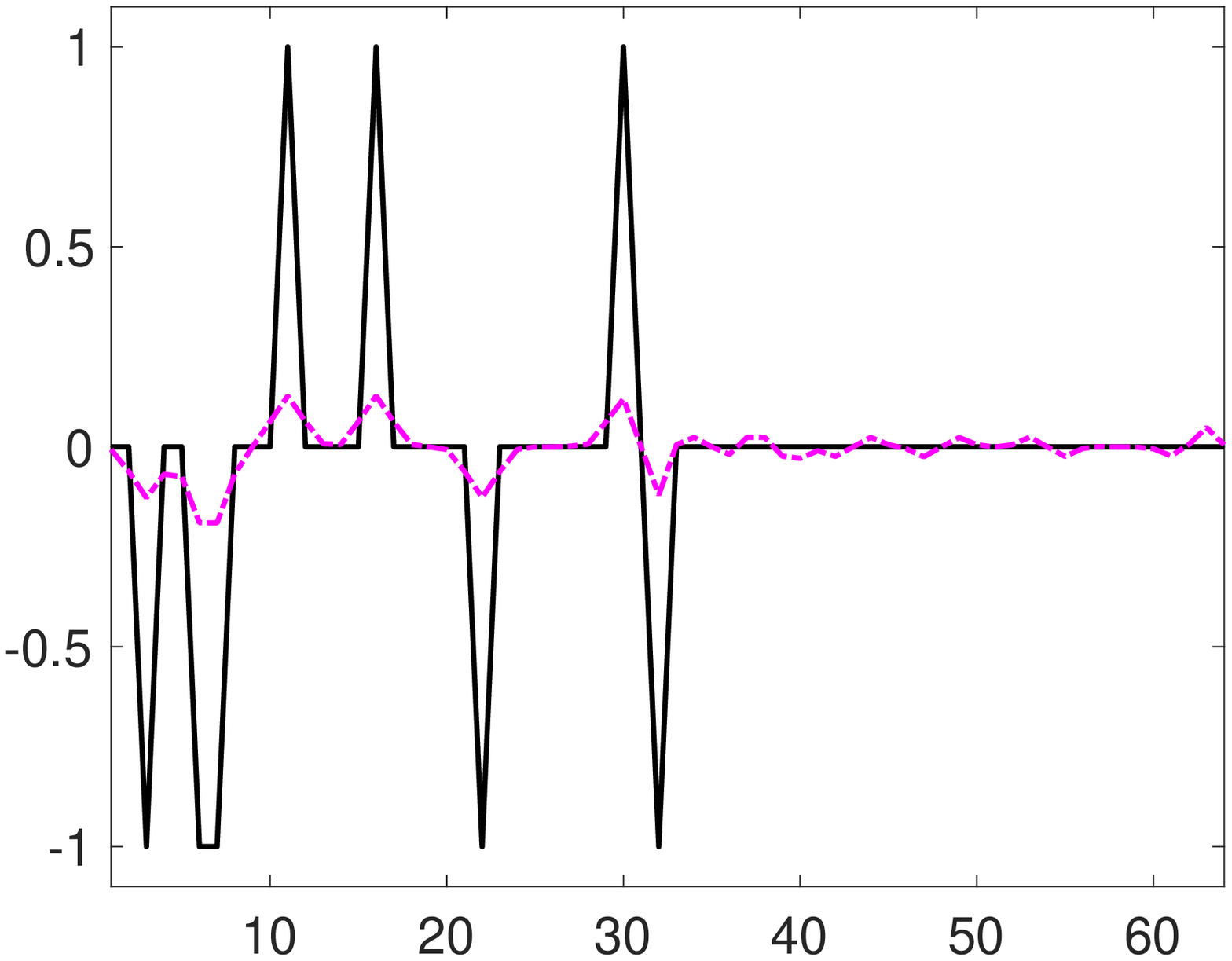}\vspace{-0.2cm}\\
\small{(a)} & \small{(b)}\\
\hspace{-0.7cm}\includegraphics[width=6.3cm]{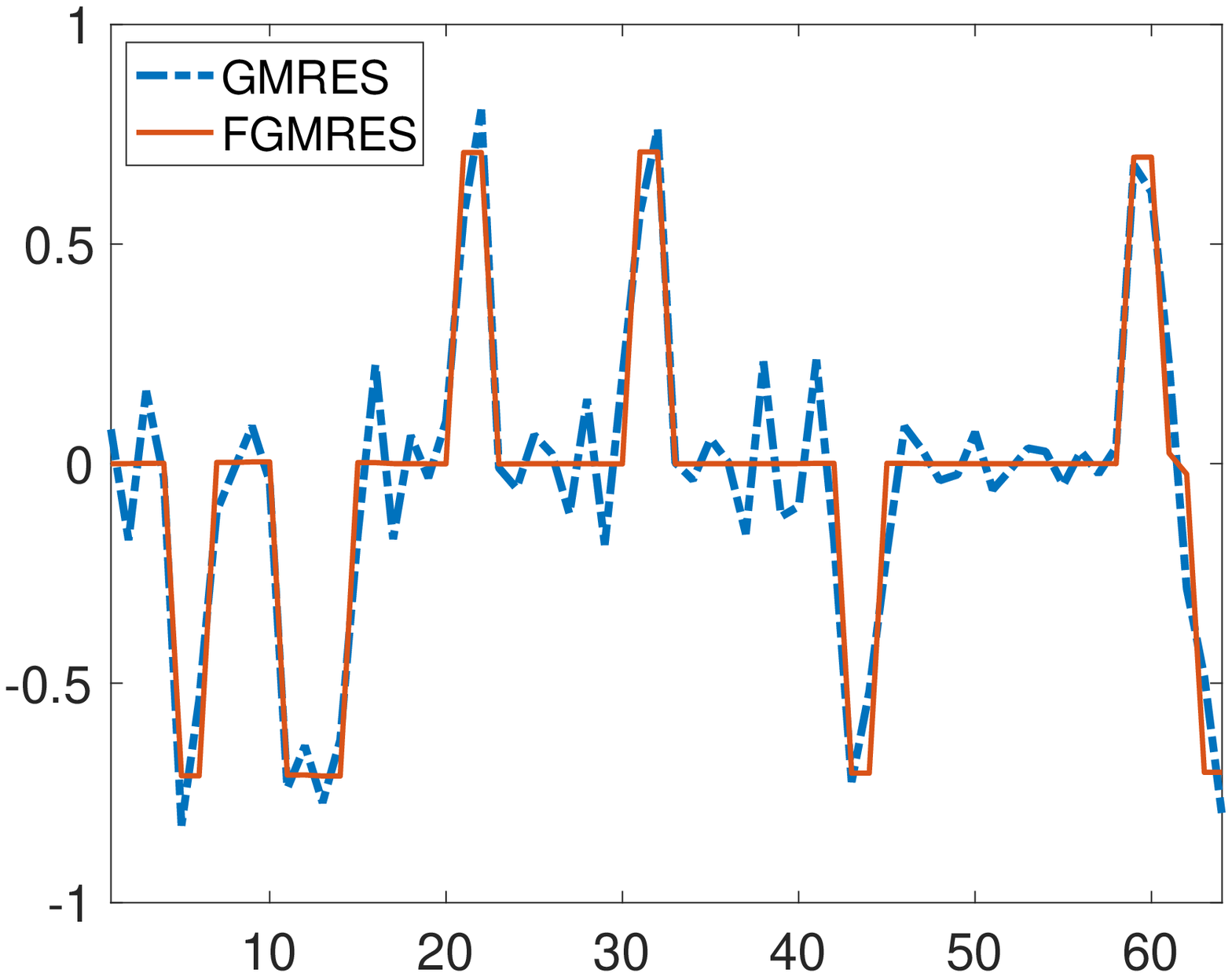} &
\hspace{-0.7cm}\includegraphics[width=6.3cm]{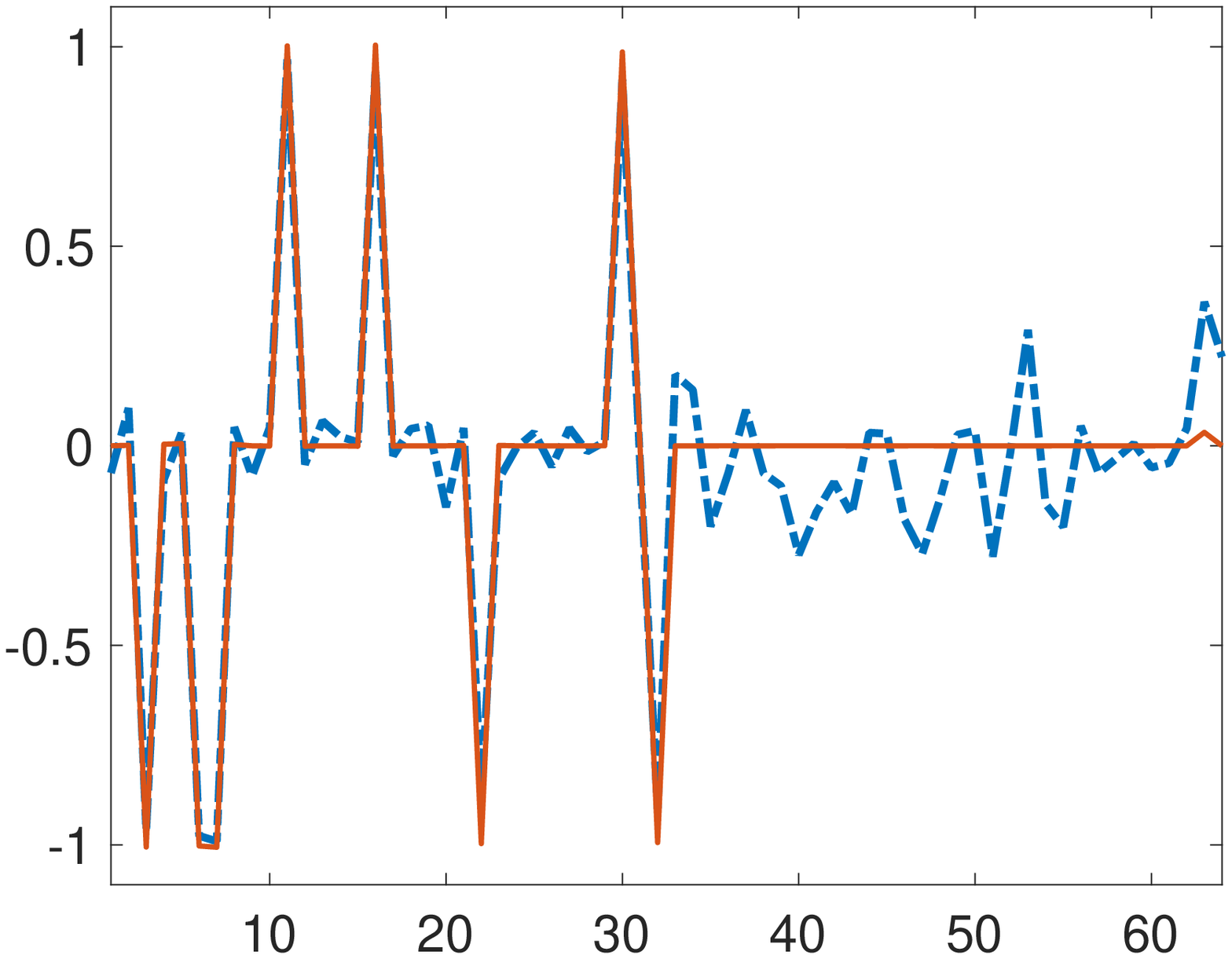}\vspace{-0.2cm}\\
\small{(c)} & \small{(d)}\\
\hspace{-0.7cm}\includegraphics[width=6.3cm]{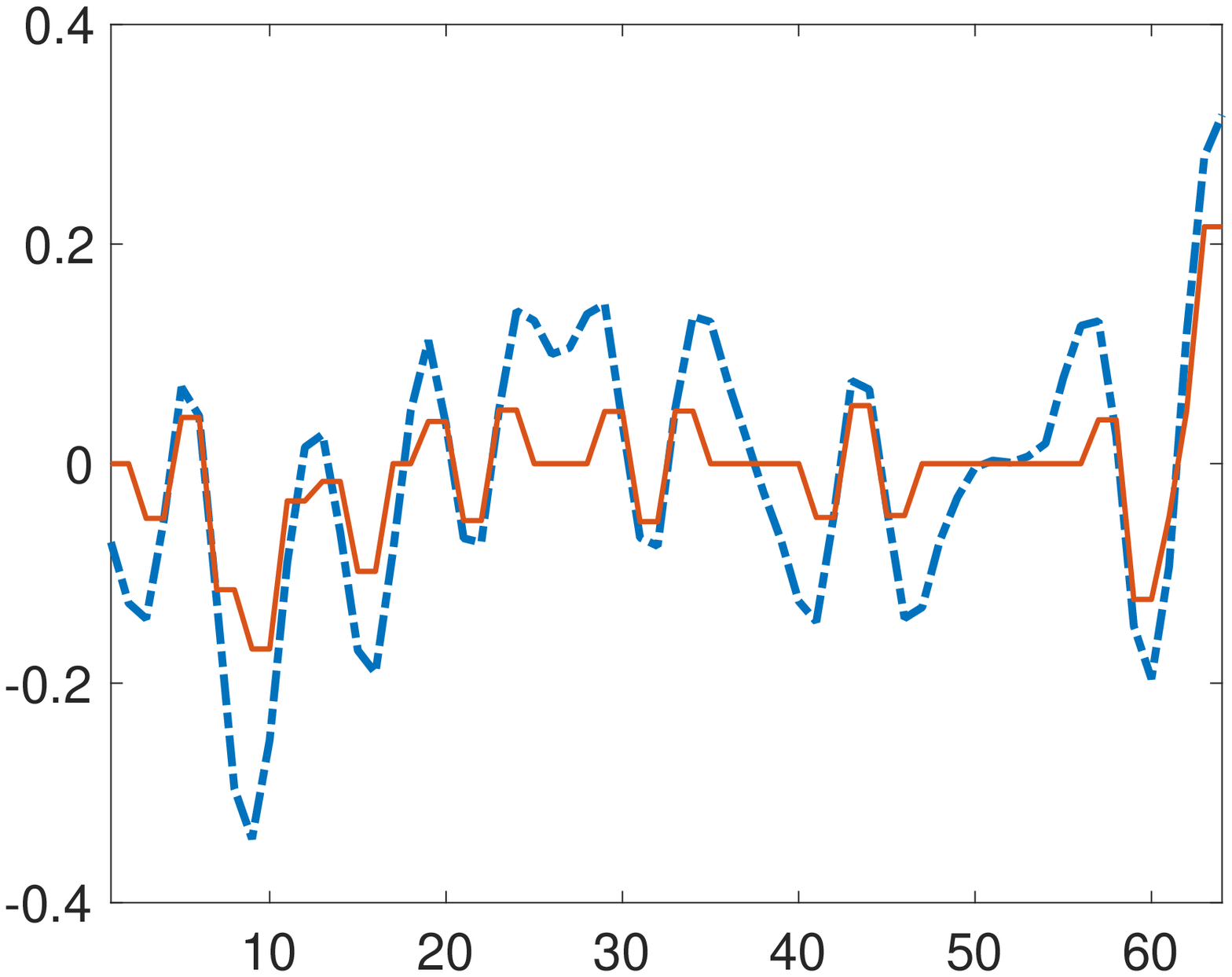} &
\hspace{-0.7cm}\includegraphics[width=6.3cm]{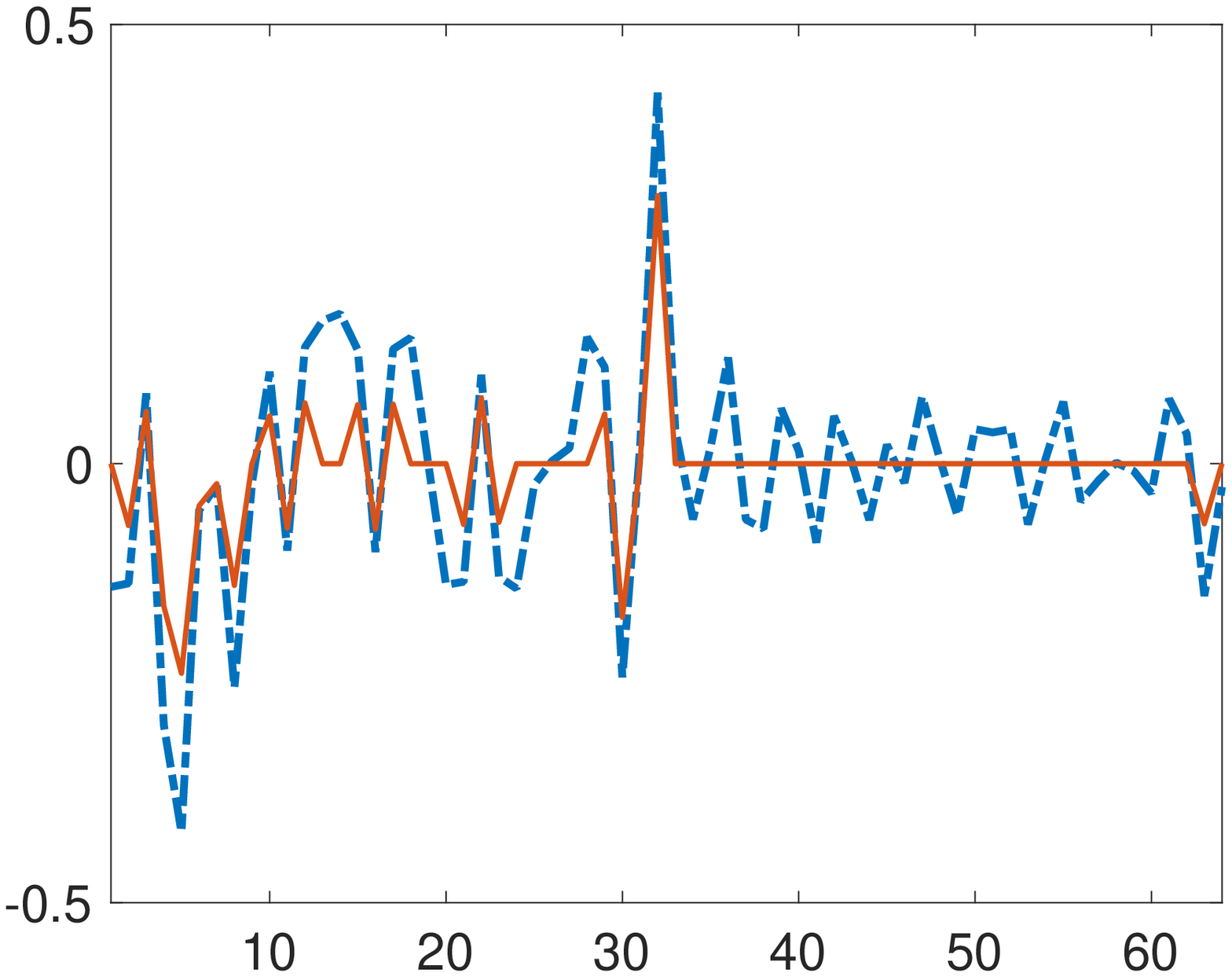}\vspace{-0.2cm}\\
\small{(e)} & \small{(f)}\\
\hspace{-0.7cm}\includegraphics[width=6.3cm]{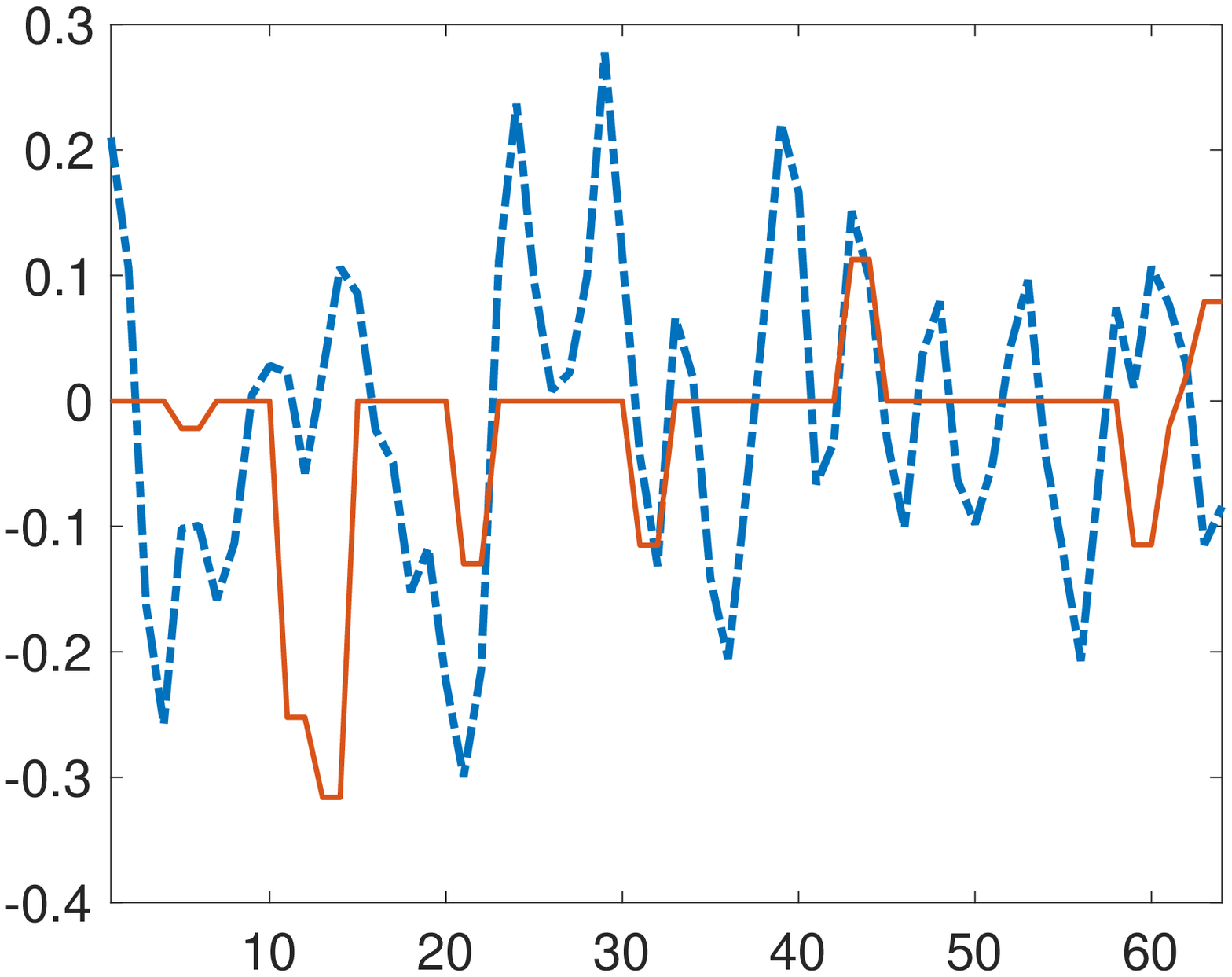} &
\hspace{-0.7cm}\includegraphics[width=6.3cm]{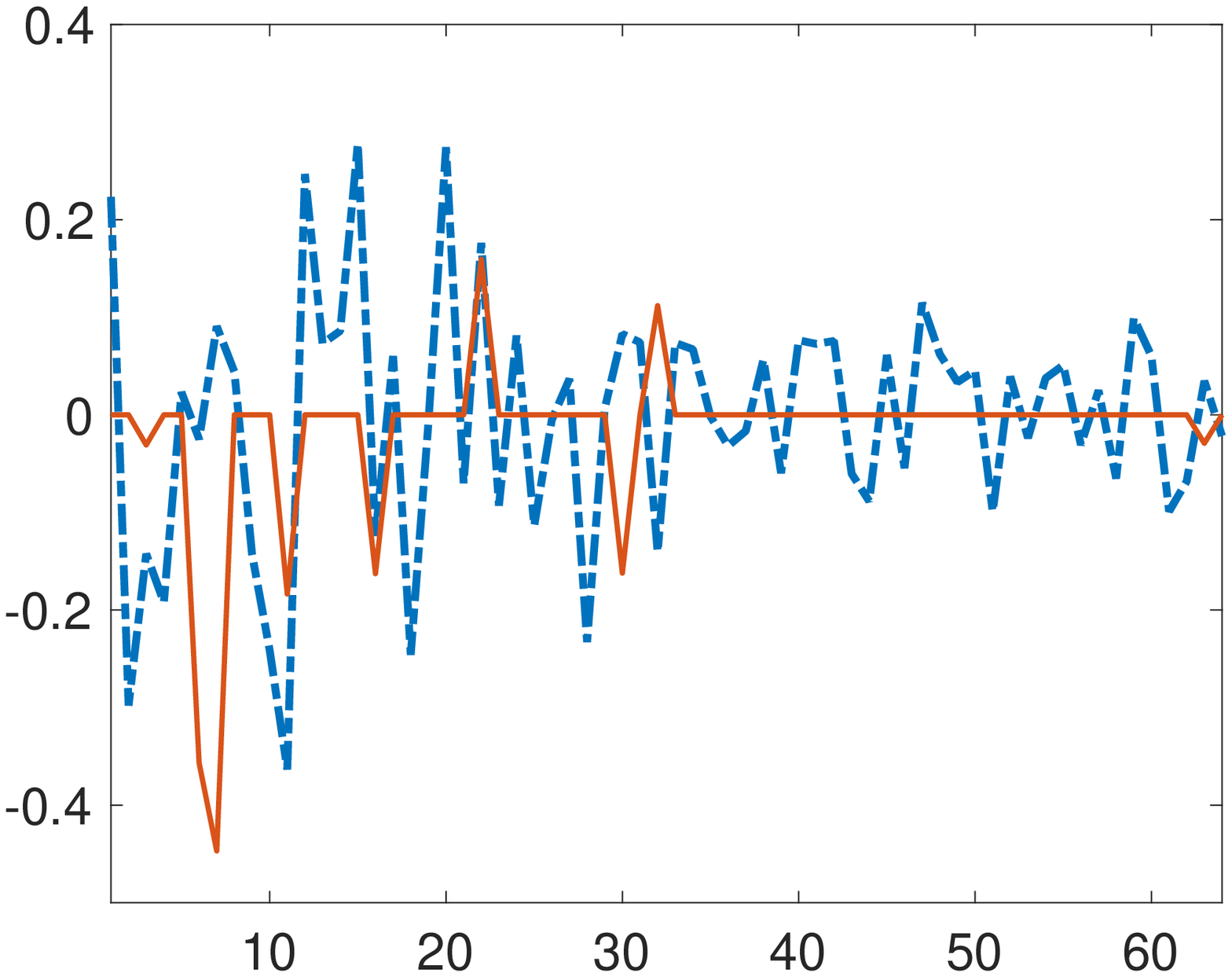}\vspace{-0.2cm}\\
\small{(g)} & \small{(h)}
\end{tabular}
\caption{1D signal deblurring and denoising problem. The right column displays the 1D Haar wavelet coefficients of the signals displayed in the left columns. The first row shows the exact and corrupted signals. The second row shows the best reconstructions obtained by GMRES (dash-dot lines) and FGMRES (solid lines). The third and fourth row show the 2nd and 4th basis vectors computed by GMRES (dash-dot lines) and FGMRES (solid lines), respectively.}
\label{fig:BasisVectors}
\end{figure}

\section{Numerical Results}
\label{sec:numerics}

In this section, we provide three experiments to demonstrate the performance of the flexible Krylov hybrid methods on various examples from image processing.  The first two experiments are examples from image deblurring, where enforcing sparsity in the {image} and sparsity in {the wavelet coefficients} are investigated separately.  The third experiment is interested in tomographic reconstruction from undersampled data.

All experiments were performed in MATLAB 2017a, using codes available in the \emph{Restore Tools} \cite{nagy2004iterative} and \emph{AIR Tools II} \cite{hansen2018air} software packages. In all presented results, relative errors are computed as $\norm[2]{\bfx_{\rm true} - \bfx_k}/\norm[2]{\bfx_{\rm true}}$.

\paragraph{Experiment 1}  In this experiment, we use an image deblurring example from atmospheric imaging where the true image, the point spread function (PSF), and the observed blurred image are provided in Figure~\ref{fig:deblurring1}. For this problem, Gaussian white noise is added to the blurred image, such that the noise level is
{$\|\bfe\|_2/\|\bfb_{\rm true}\|_2 = 5\times 10^{-2}$}, where $\bfb_{\rm true}=\bfb-\bfe$.
\begin{figure}[bthp]
 \begin{center}
   \includegraphics[width=.7\textwidth]{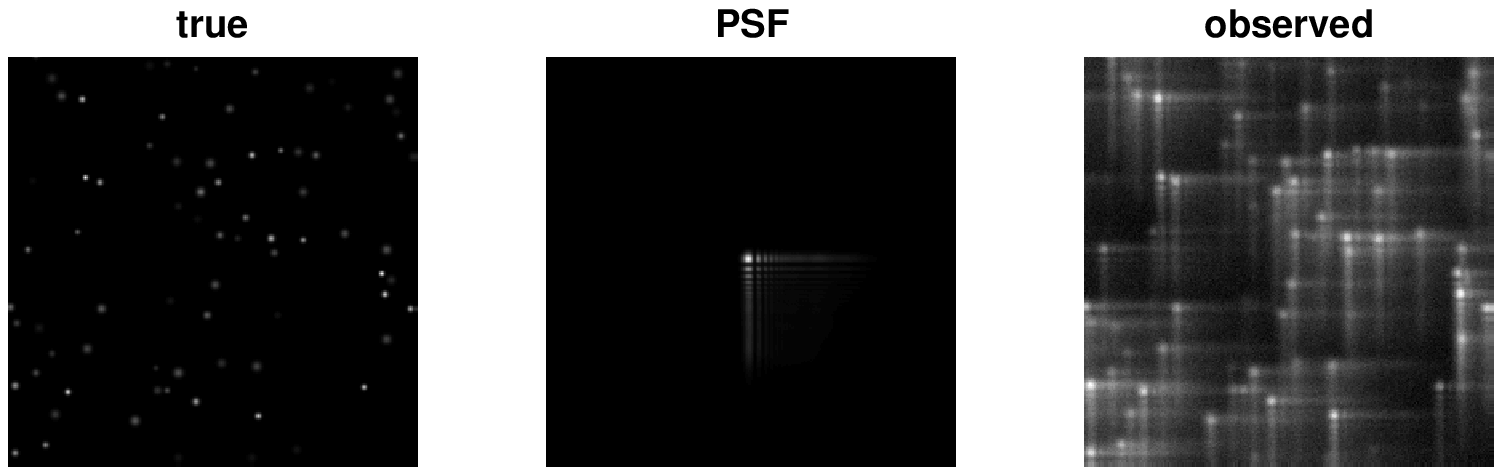}
  \end{center}
  \caption{Experiment 1: Image deblurring example.  Here we show the true image, the point spread function, and the observed blurred and noisy image.  The size of the image is $256 \times 256$ pixels}\label{fig:deblurring1}
\end{figure}

For the reconstructions, we assumed reflexive boundary conditions and solved the $\ell_1$-regularized problem with $\bfPsi = \bfI$, which is appropriate because the desired image is quite sparse.  First we provide a comparison of various Golub-Kahan-based methods.  In Figure~\ref{fig:rel1}, we provide relative errors per iteration for flexible methods described in Section~\ref{sec:flexible}, namely FLSQR, FLSQR-I and FLSQR-R with automatic regularization parameter selection using the ``secant update'' discrepancy principle.  Relative reconstruction errors for LSQR are provided for comparison.  Similarly to the observations made in Section~\ref{sec:flexible}, the flexible methods exhibit faster convergence to more accurate solutions than the standard approach. Furthermore, we see that the flexible hybrid methods are able to stabilize semiconvergent behavior by selecting an appropriate regularization parameter, and the stopping criterion (still based on the ``secant update'' strategy) seems to work well.
\begin{figure}[bthp]
 \begin{center}
   \includegraphics[width=.7\textwidth]{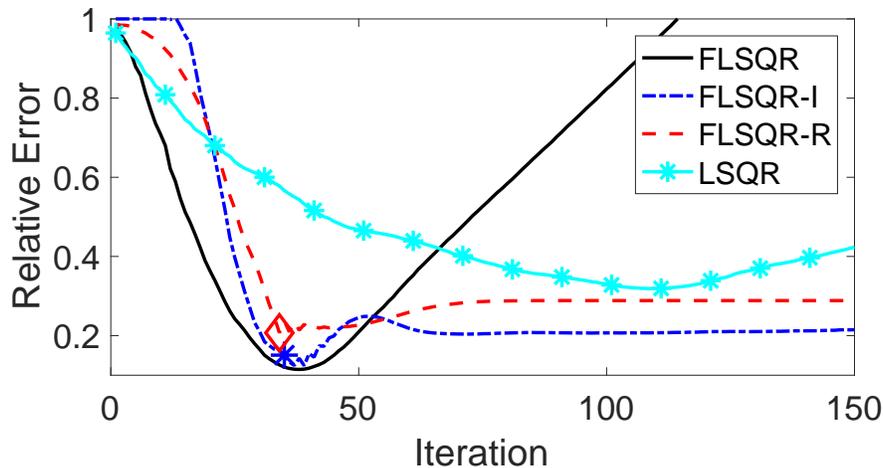}
  \end{center}
  \caption{Experiment 1: Comparison of relative reconstruction errors.  Regularization parameter $\lambda$ is selected automatically using the modified discrepancy principle for FLSQR-I and FLSQR-R. $\lambda=0$ for FLSQR and LSQR. Automatically determined stopping iterations for the hybrid approaches are denoted by the diamond and star.}\label{fig:rel1}
\end{figure}

In Figure~\ref{fig:basis}, we provide the basis images for FLSQR-R and LSQR for \linebreak[4]$k = 10, 20, 100$. Note that basis images for FLSQR-R correspond to the {FGK} vectors, while the LSQR ones correspond to the standard {GKB} vectors.
\begin{figure}[bthp]
 \begin{center}
   \includegraphics[width=.7\textwidth]{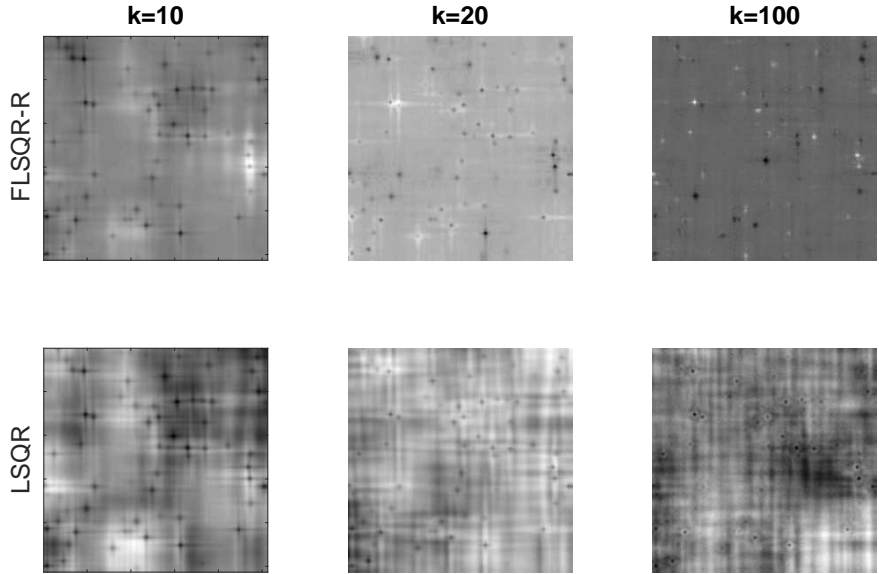}
  \end{center}
  \caption{Experiment 1: Basis images for FLSQR-R and LSQR for $k = 10, 20, 100$.  These are solution vectors (i.e., $\bfz_k$ for FLSQR-R) that have been reshaped into images.}\label{fig:basis}
\end{figure}
It is evident that the basis images for the flexible method are able to better capture the flat regions of the image.  Also, for large $k,$ the FLSQR-R basis {image} is not affected by the noise amplification that is present in the LSQR basis image.  Thus, we expect that by constructing a better solution basis (i.e., one that is not as affected by noise and one that can capture sparsity properties of the image), the flexible methods can be successful for sparse image reconstruction.

Next, we investigate some parameter choice methods. In Figure~\ref{fig:rel_param}, we provide relative reconstruction errors for FLSQR-R and `FLSQR-R dp'.  Both of these methods use the discrepancy principle to obtain the regularization parameter, where FLSQR-R utilizes the ``secant update'' parameter choice method described in~\cite{gazzola2014generalized} {and `FLSQR-R dp' enforces the discrepancy principle to be satisfied at each iteration}, which require prior knowledge of the noise level. Relative errors for `FLSQR-R opt' correspond to selecting the regularization parameter at each iteration that minimizes the error of the current iterate to the true solution.  It is worth noting that, since the basis vectors are generated based on the current solution (because of flexibility), this approach does not necessarily produce the best overall regularization parameter for the problem.
\begin{figure}[bthp]
 \begin{center}
   \includegraphics[width=.7\textwidth]{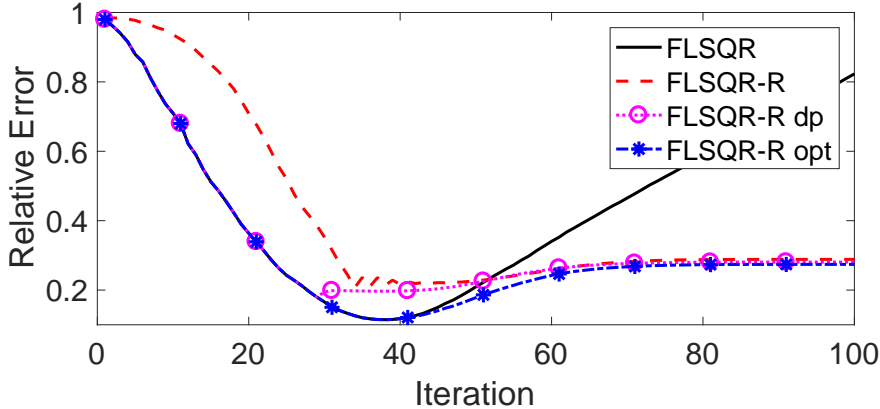}
  \end{center}
  \caption{Experiment 1: Relative reconstruction errors for different parameter choice methods.  FLSQR-R and FLSQR-R dp use the discrepancy principle, and thus require an estimate of the noise level. FLSQR-R uses a modified discrepancy principle. FLSQR-R opt corresponds to selecting the optimal regularization parameter at each iteration{, which is not necessarily the overall best parameter because of flexibility}.}\label{fig:rel_param}
\end{figure}

Finally, we compare our approach to other methods for solving the $\ell_1$-regularized problem.  In Figure~\ref{fig:rel2}, we provide relative reconstruction errors for GAT \cite{gazzola2014generalized}, {PIRN}, FISTA \cite{beck2009fast}, and SpaRSA \cite{wright2009sparse}, with FLSQR-R provided from Figure~\ref{fig:rel1} for comparison.  Since the regularization parameter for {PIRN,} FISTA, and SpaRSA must be selected prior to execution, we used the regularization parameter that was selected by FLSQR-R when the stopping criterion was satisfied.  We note that both FISTA and SpaRSA compute reconstructions with similar or slightly better accuracy than FLSQR-R, but two main advantages of the hybrid approaches are that the regularization parameter can be selected automatically, and the reconstruction can be obtained in fewer iterations.  The main cost per iteration for all of these methods is one matrix-vector multiplication with $\bfA$ and one with $\bfA\t$.

\begin{figure}[bthp]
 \begin{center}
   \includegraphics[width=.7\textwidth]{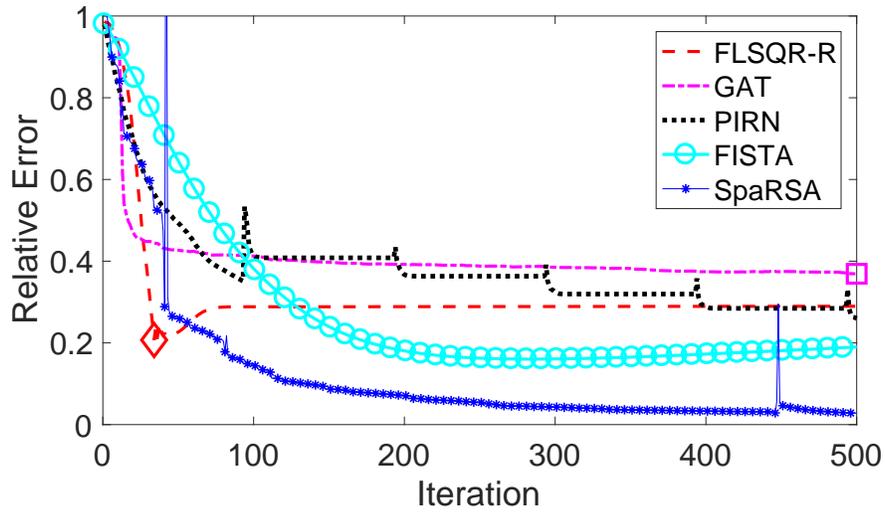}
  \end{center}
  \caption{Experiment 1: Relative reconstruction errors are provided to compare the flexible GK methods to some previously proposed methods.  It is important to note that we have selected the PIRN, FISTA, and SpaRSA regularization parameter using our FLSQR-R approach.}\label{fig:rel2}
\end{figure}

\paragraph{Experiment 2}
In this experiment, we investigate the transformed $\ell_1$-regularized problem for an image deblurring example.  For this problem, we use the cameraman image shown in Figure~\ref{fig:Ex2_deblur}, where out of focus blur of radius 4 and {Gaussian white noise with} noise level $0.01$ {are considered}.
Although a wide range of transformations $\bfPsi$ can be employed, for simplicity we used a 2D Haar wavelet decomposition with 3 levels.  For this example, the images contain $256 \times 256$ pixels and the image itself is not sparse (having only $27$ zero pixels). However, the transformed true image (also provided in Figure~\ref{fig:Ex2_deblur}) has $1502$ zero pixels and thus it is appropriate to consider the transformed $\ell_1$-regularized problem.
\begin{figure}[bthp]
\begin{center}
  \includegraphics[width=.7\textwidth]{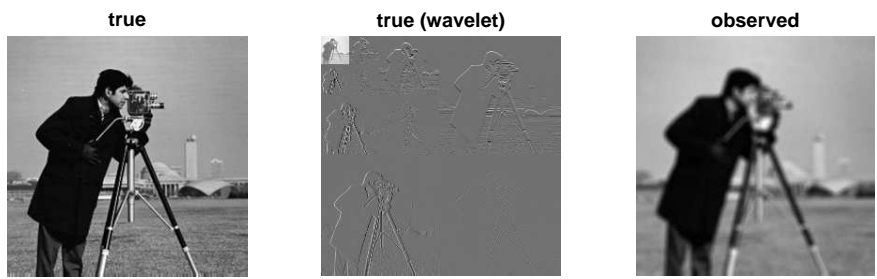}
 \end{center}
 \caption{Experiment 2: Image deblurring example.  Here we show the true image, the wavelet coefficients of the true image, and the observed image.}\label{fig:Ex2_deblur}
\end{figure}

First we investigate the Golub-Kahan-based methods.  In Figure~\ref{fig:Ex2_rel1}, we provide the relative reconstruction errors for FLSQR, FLSQR-I, and FLSQR-R, where LSQR on the original problem is provided for comparison.  In terms of relative error, the flexible methods take a few more iterations and provide slightly smaller reconstruction errors, but the difference is more pronounced in the reconstructions.  Subimages of the reconstructions are provided in Figure~\ref{fig:Ex2_recon}, along with the error images.  We observe that the flexible methods are able to better capture the flat regions of the image.
\begin{figure}[bthp]
 \begin{center}
   \includegraphics[width=.7\textwidth]{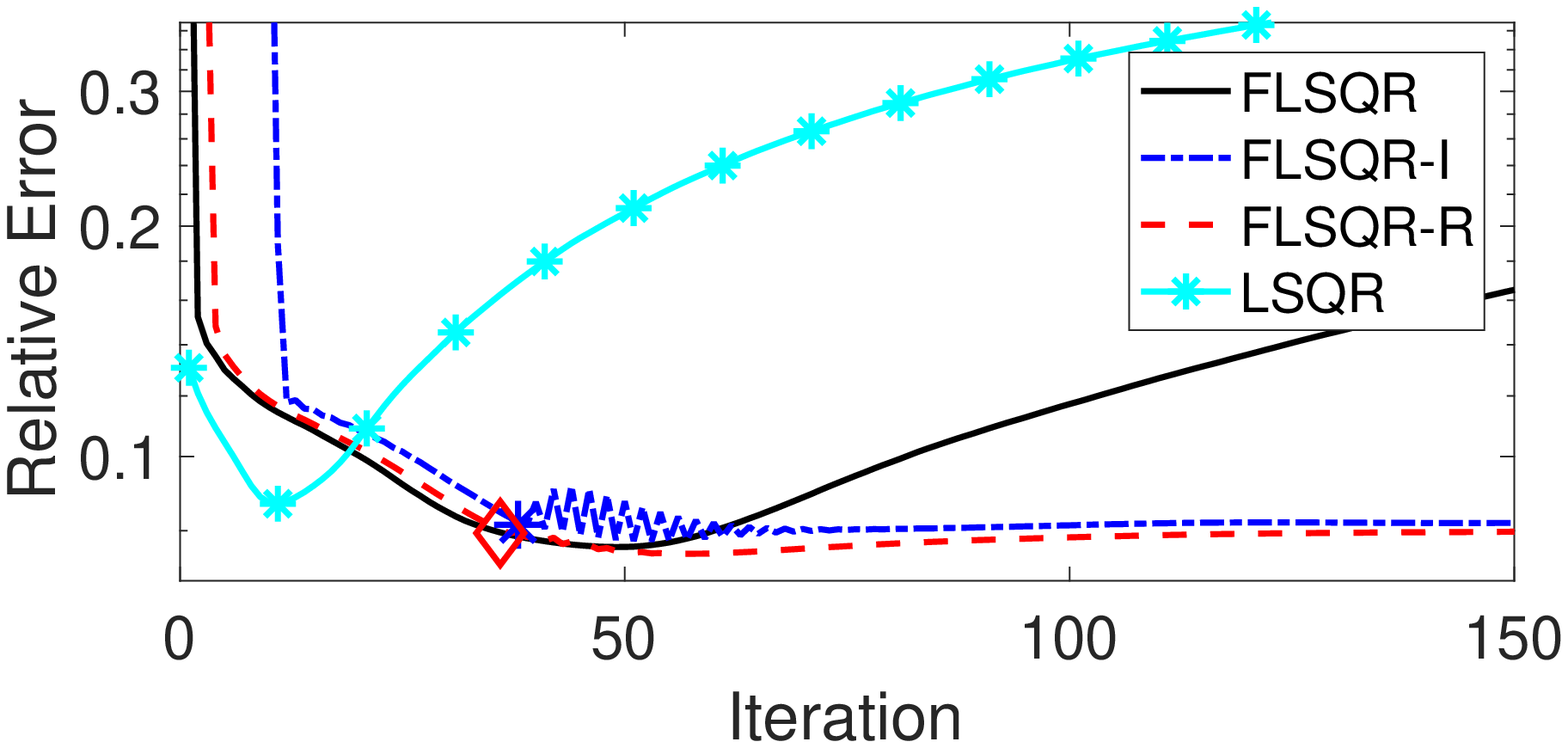}
  \end{center}
  \caption{Experiment 2: Relative reconstruction errors for Golub-Kahan-based approaches.  Regularization parameter $\lambda$ is selected automatically using the discrepancy principle for FLSQR-I and FLSQR-R. $\lambda=0$ for FLSQR and LSQR.}
  \label{fig:Ex2_rel1}
\end{figure}

\begin{figure}[bthp]
 \begin{center}
   \includegraphics[width=\textwidth]{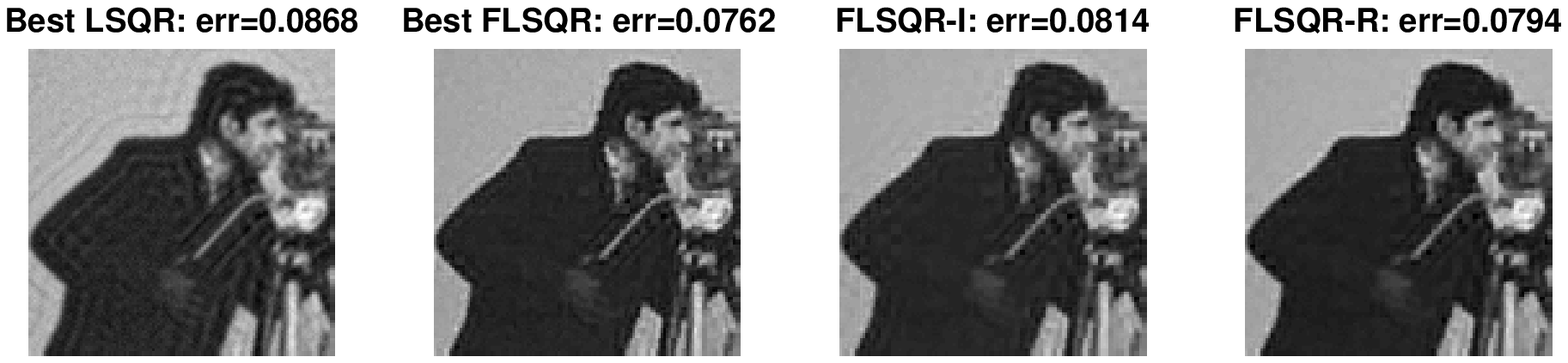}\\
      \includegraphics[width=\textwidth]{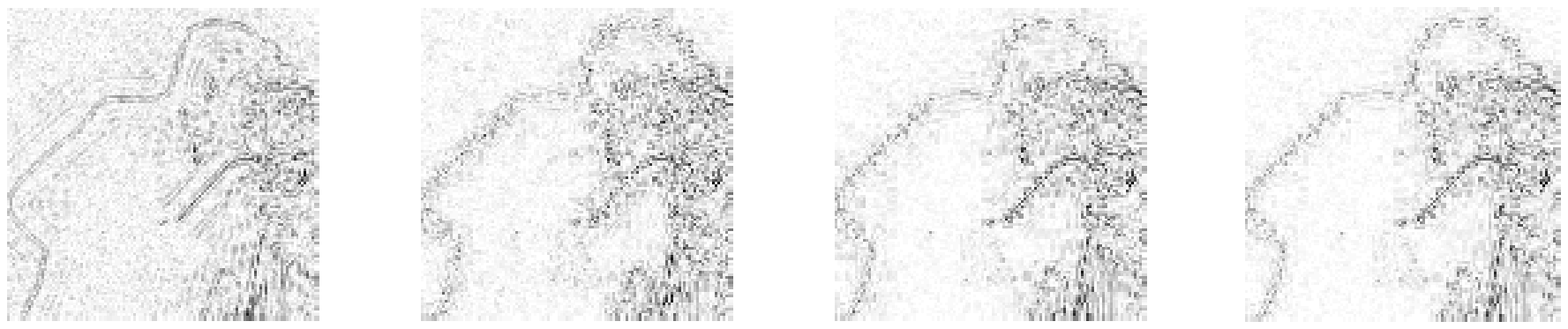}
  \end{center}
  \caption{Experiment 2: Sub-images of reconstructions for Golub-Kahan-based methods, along with absolute error images in inverted colormap (where white corresponds to small error), are provided for LSQR, FLSQR, FLSQR-I, and FLSQR-R. Relative reconstruction errors for the entire image are provided in the titles.}
  \label{fig:Ex2_recon}
  \end{figure}

Next we compare FLSQR-R to the GAT method applied to the transformed problem, as well as to FISTA on the transformed problem, with the regularization parameter computed from FLSQR-R.  Here, the computed parameter is $2.4\times 10^{-2}$ and is too small.  Thus, we also provide in `FISTA opt' the results for FISTA with the optimal regularization parameter $0.1,$ which was determined by searching over $10$ logarithmically spaced values between $10^{-3}$ and $1$, and selecting the one with the smallest reconstruction error.  We observe that only for a good choice of the regularization parameter FISTA reconstructions are similar to ours while, for poor choices of the regularization parameter, FISTA reconstructions are either too blocky or contaminated with noise.

\begin{figure}[bthp]
 \begin{center}
   \includegraphics[width=.7\textwidth]{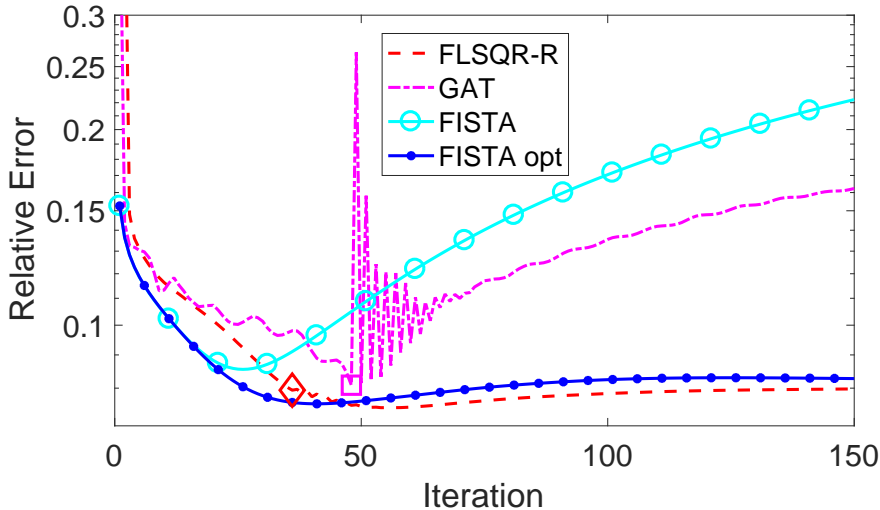}
  \end{center}
  \caption{Experiment 2: Relative reconstruction errors are provided to compare the flexible GK methods to some {existing methods.  FISTA uses the regularization parameter  selected by FLSQR-R, and FISTA opt uses a regularization parameter that was found empirically using the true image.}}
\end{figure}

\paragraph{Experiment 3}
We consider a sparse X-ray tomographic reconstruction example, with undersampled data. The goal of this experiment is to assess how the new solvers based on the FGK decomposition perform for the solution of the transformed {$\ell_1$-regularized problem} (\ref{eq:pnormAb})
where
$\bfA$ is underdetermined and $\bfPsi$ represents 2D Haar wavelet transform with 4 levels. In \cite{CSCT} it has been empirically shown that the compressive sensing theory applies when performing standard structured undersampling patterns
and when solving either the $\ell_1$ or the total variation regularized problems. The test problem considered here is generated using the \texttt{paralleltomo} function from \emph{AIR Tools II} \cite{hansen2018air}, which models a 2D equidistant parallel-beam scanning geometry, with the following parameters:
\[
\text{\texttt{N = 256, theta = 0:2:179, p = round(sqrt(2)*N), d = sqrt(2)*N} .}
\]
This computes a fairly underdetermined sparse matrix $\bfA$ of size $32580\times 65536$ (which roughly corresponds to 50\% undersampling). The exact solution $\bfx$ is a vectorialization of the well-known Shepp-Logan phantom of size $256\times 256$ pixels. The transformed exact solution $\bfPsi\bfx$ only has $27492$ nonzero entries (which roughly corresponds to 60\% sparsity). Note that, with such undersampling and sparsity, and according to \cite{CSCT}, recovery should be experimentally guaranteed. Gaussian white noise of level $10^{-2}$ is added to the exact data.

Figure \ref{fig:ex3_relerr_iterative} displays the history of relative errors associated to different purely iterative regularization methods (i.e., with $\lambda=0$ in (\ref{eq:transformed}){)}: since we are dealing with a rectangular matrix, only LSQR and LSMR together with their  flexible versions are considered. We can clearly see the benefits of introducing flexibility into the solution subspaces: indeed, a greater accuracy is achieved by the flexible methods (with a computational cost comparable to the standard solvers), together with a less pronounced semiconvergence (this is particularly true for FLSMR, in accordance to the observations in \cite{chung2015hybrid}).
\begin{figure}
\includegraphics[width=\textwidth]{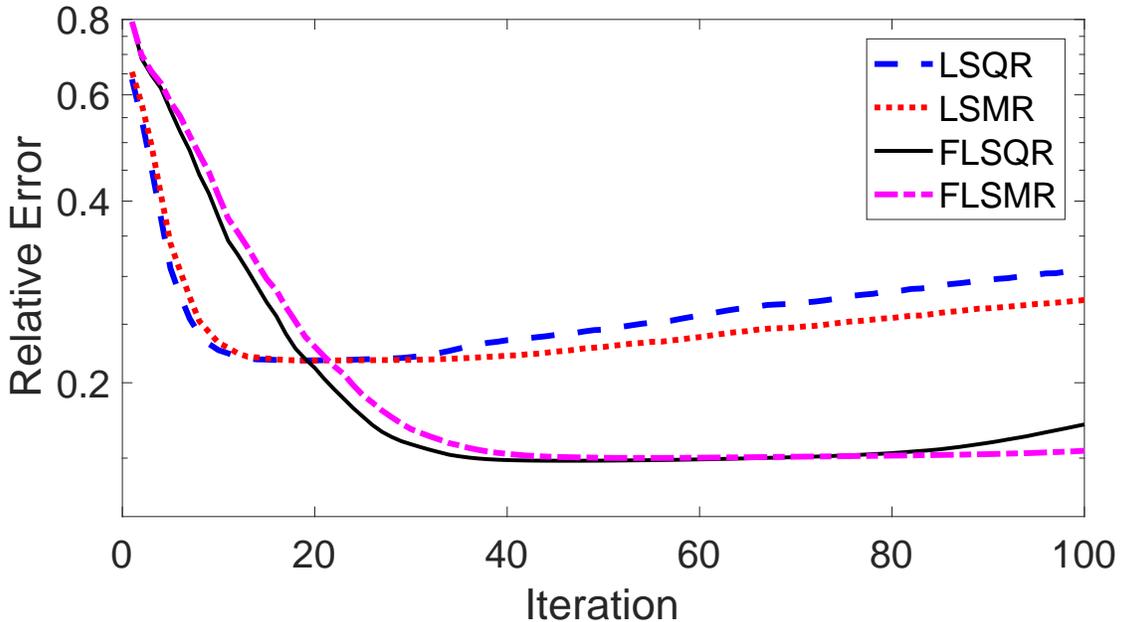}
\caption{Experiment 3: History of the relative errors, considering purely iterative methods.}\label{fig:ex3_relerr_iterative}
\end{figure}
Figure \ref{fig:ex3_relerr_hybrid} displays the history of the relative errors when the {`FLSQR-I dp'} method is employed (with the regularization parameter chosen at each iteration by the discrepancy principle), and compares it to other solvers for (\ref{eq:transformed}). In particular, we run FISTA with a ``standard'' stepsize choice (i.e., the stepsize is chosen as the Lipschitz constant $\sigma_1^{-2}$, which is estimated by running a few GKB iterations); correspondingly, the threshold is set to $\lambda\cdot\sigma_1^{-2}$. {Finally, we compare with SpaRSA, IRN, and PIRN}. As already remarked, all these well-established solvers used for comparisons require the regularization parameter $\lambda$ to be set at the beginning of the iterative process: for this experiment we choose $\lambda=4.2\cdot 10^{-5}$, which is the value computed by the classical discrepancy principle at the end of the {`FLSQR-I dp'} iterations (when also some stabilization occurred in the iteration-dependent values of the regularization parameter). We can clearly see that SpaRSA does not perform well for this problem, and that FISTA rapidly stagnates and computes solutions of lower quality with respect to the {`FLSQR-I dp'} ones. PIRN seems to be the method performing better in terms of relative errors (requiring anyway more iterations than {`FLSQR-I dp'} to reach an optimal accuracy), and it surely outperforms IRN, which is not so effective because of the small $\lambda$ considered in this framework. We do not show the behavior of the FLSQR-R, FLSMR-I, and FLSMR-R {hybrid} methods as they are very similar to the FLSQR-I method for this problem.
\begin{figure}
\includegraphics[width=\textwidth]{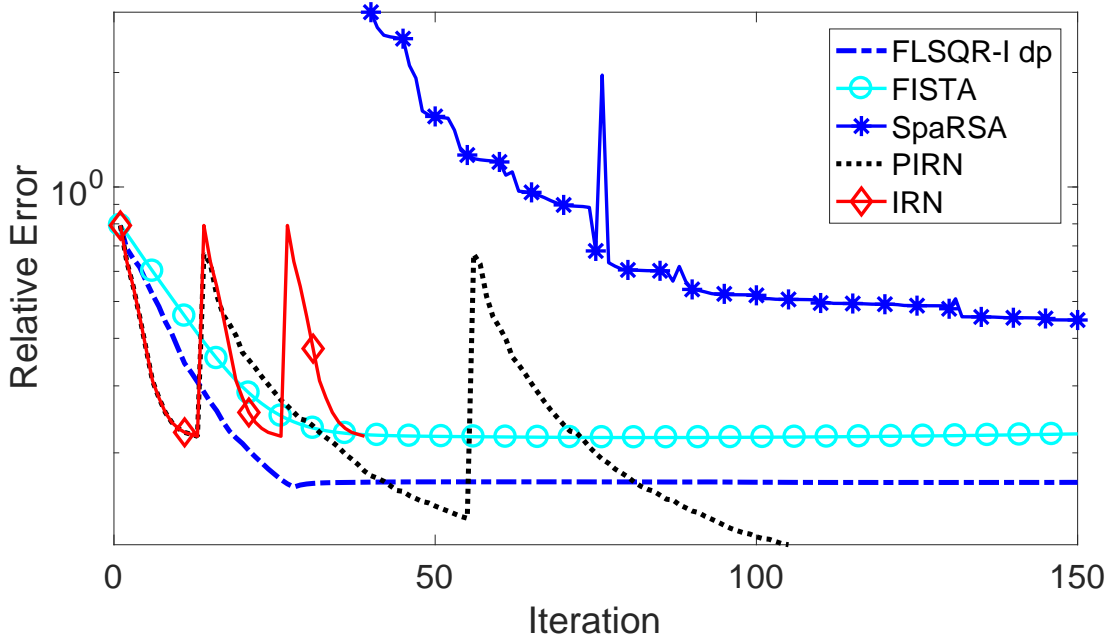}
\caption{Experiment 3: History of the relative errors, comparing the FLSQR-I method to FISTA, SpaRSA, the standard IRN, and the preconditioned PIRN.}\label{fig:ex3_relerr_hybrid}
\end{figure}
\begin{figure}[tbp]
\centering
\begin{tabular}{ccc}
\hspace{-0.7cm}\small{\textbf{exact}} &
\hspace{-0.7cm}\small{\textbf{FLSQR-I}} &
\hspace{-0.7cm}\small{\textbf{FISTA}}\\
\hspace{-0.7cm} &
\hspace{-0.7cm}\small{(0.1626, \# 28)}&
\hspace{-0.7cm}\small{(0.2194, \# 82)}\vspace{-0.05cm}\\
\hspace{-0.7cm}\includegraphics[width=4.5cm]{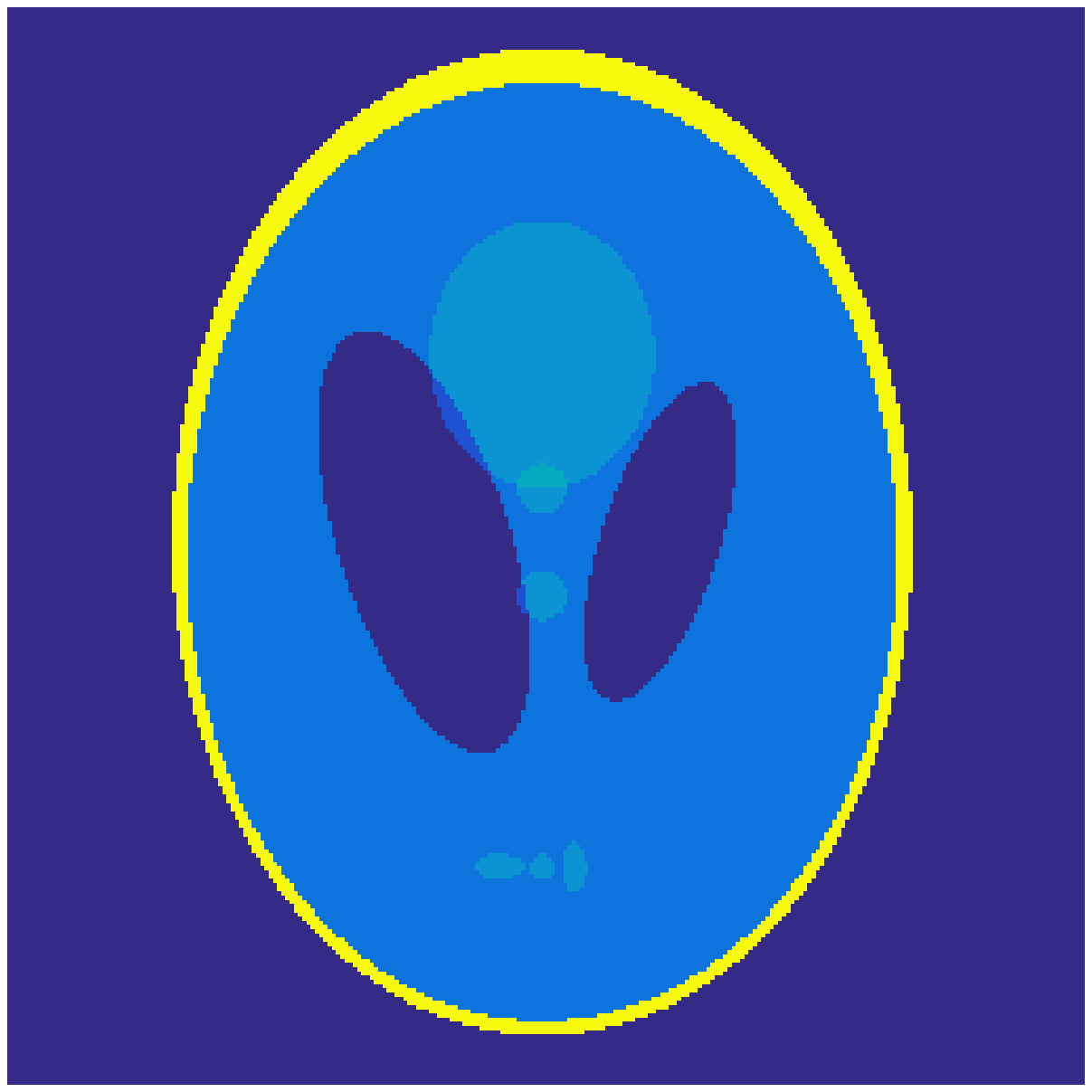} &
\hspace{-0.7cm}\includegraphics[width=4.5cm]{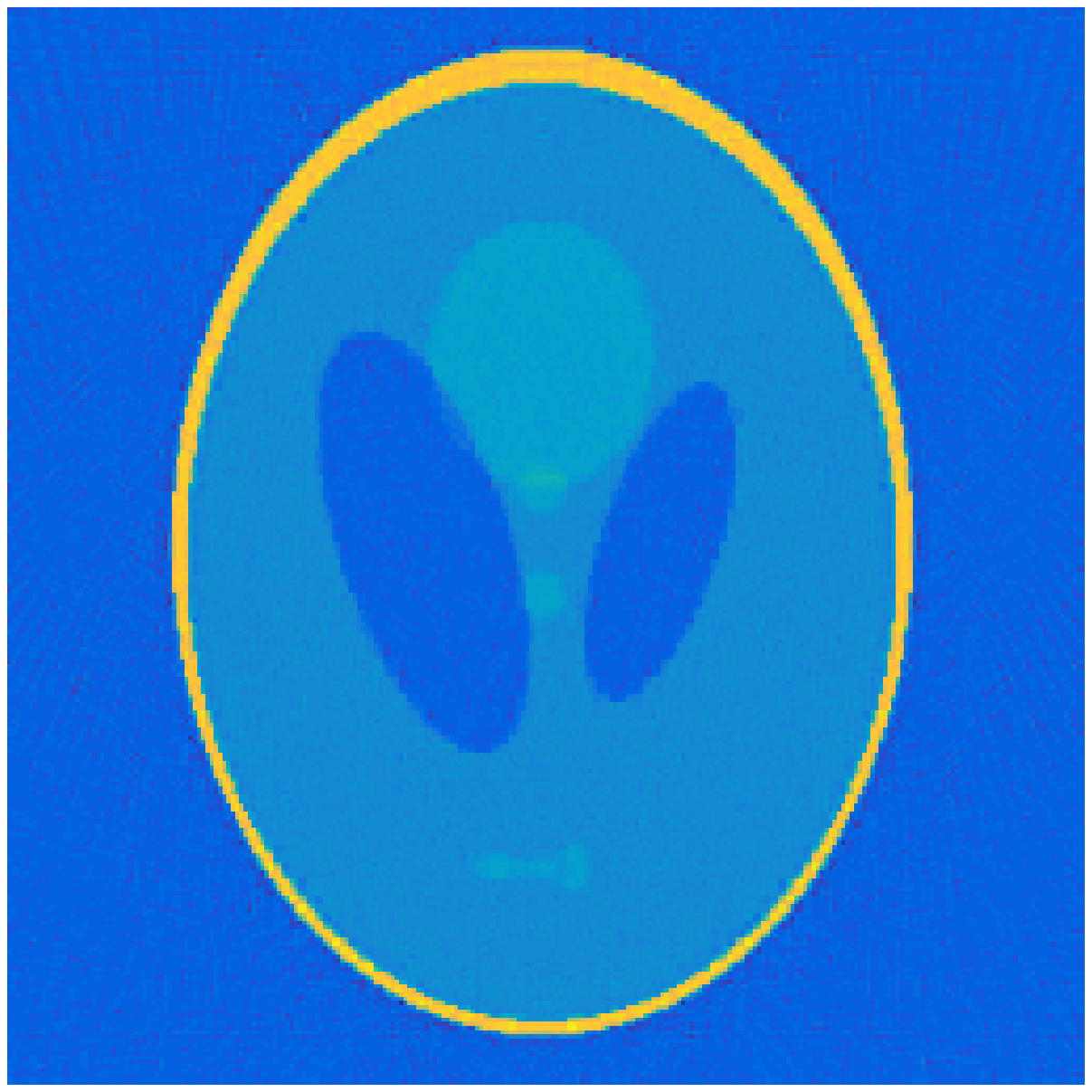} &
\hspace{-0.7cm}\includegraphics[width=4.5cm]{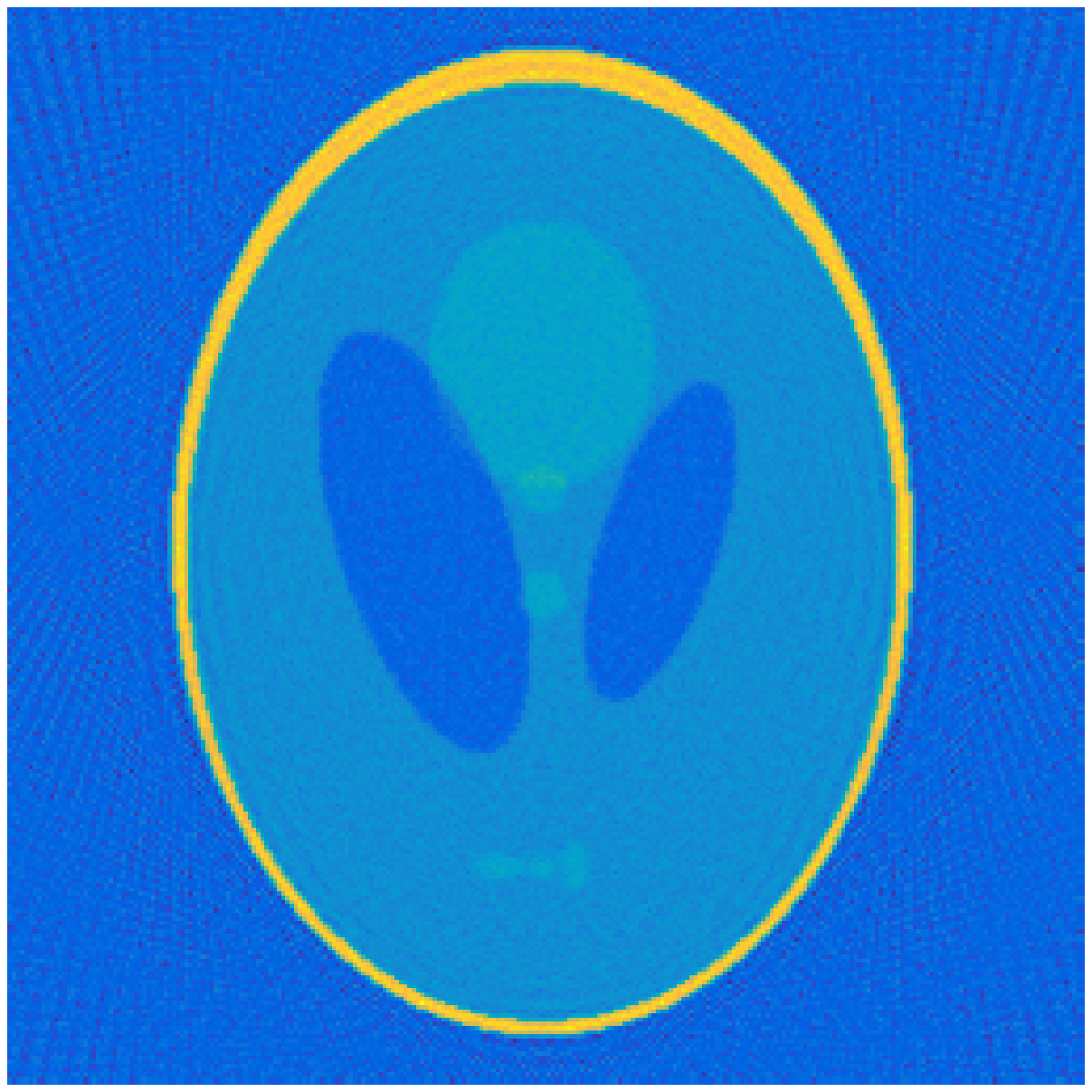}\vspace{-0.2cm}\\
\hspace{-0.7cm}\small{\textbf{SpaRSA}} &
\hspace{-0.7cm}\small{\textbf{IRN}} &
\hspace{-0.7cm}\small{\textbf{PIRN}}\\
\hspace{-0.7cm}\small{(0.4467, \# 150)} &
\hspace{-0.7cm}\small{(0.2211, \# 26)} &
\hspace{-0.7cm}\small{(0.1150, \# 105)}\vspace{-0.05cm}\\
\hspace{-0.7cm}\includegraphics[width=4.5cm]{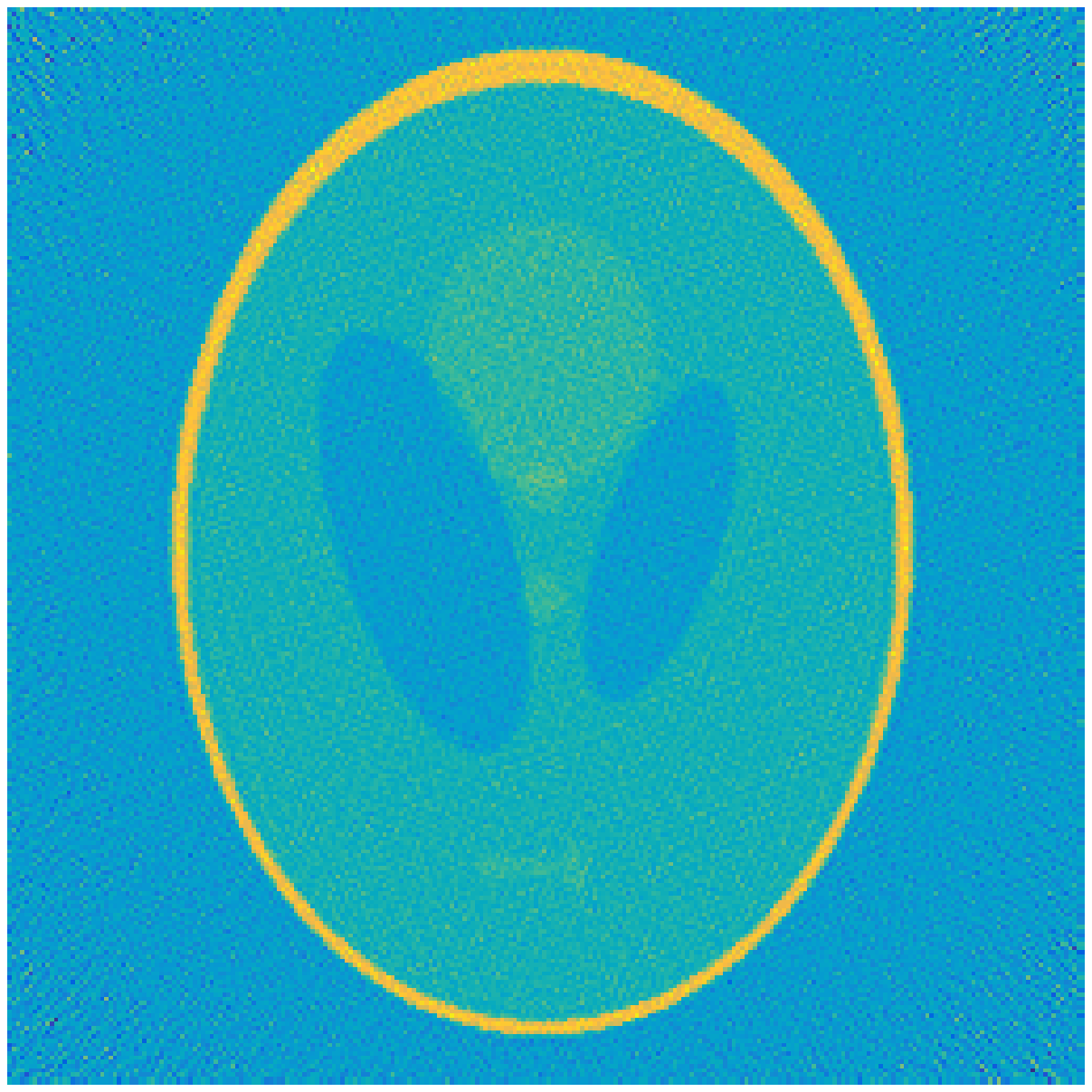} &
\hspace{-0.7cm}\includegraphics[width=4.5cm]{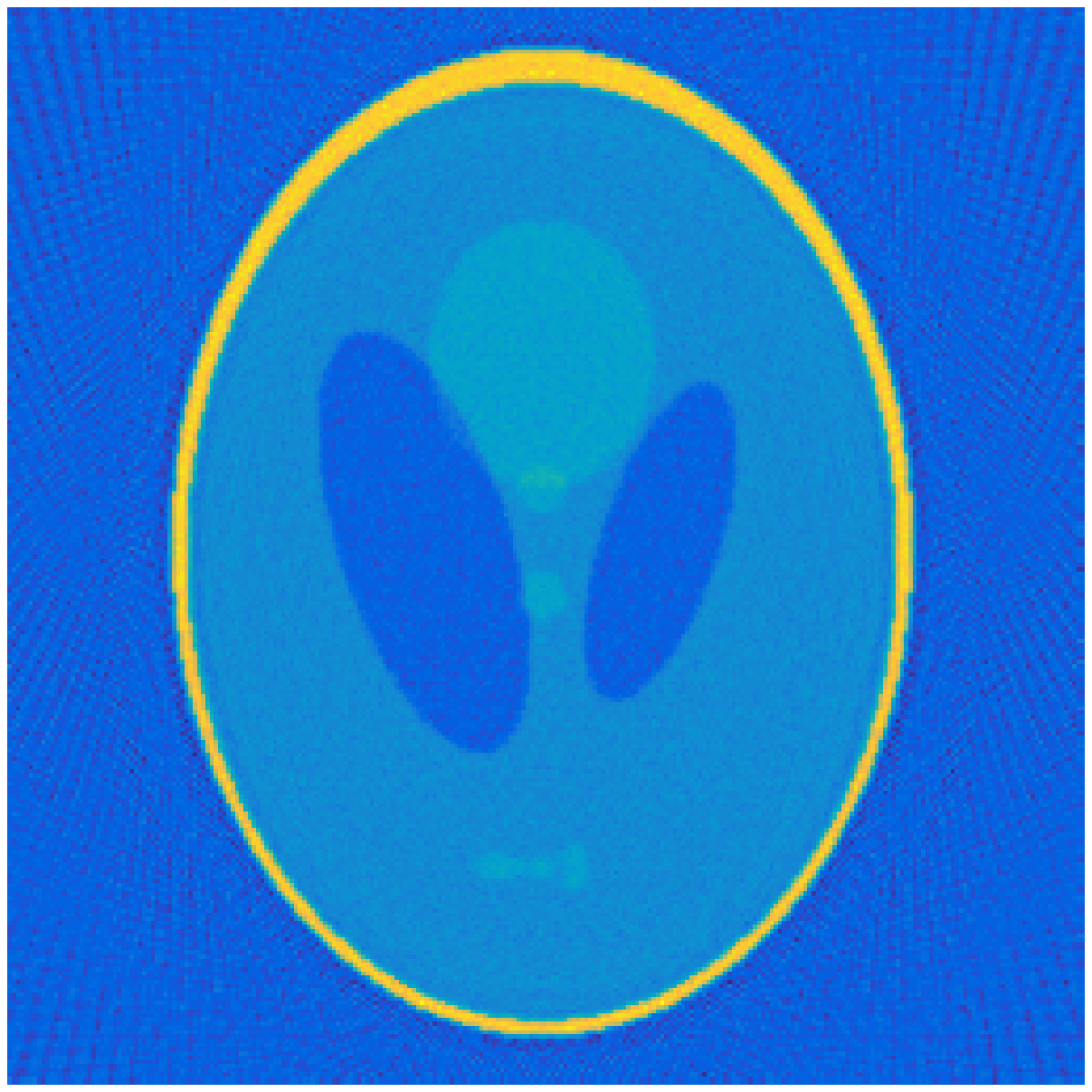} &
\hspace{-0.7cm}\includegraphics[width=4.5cm]{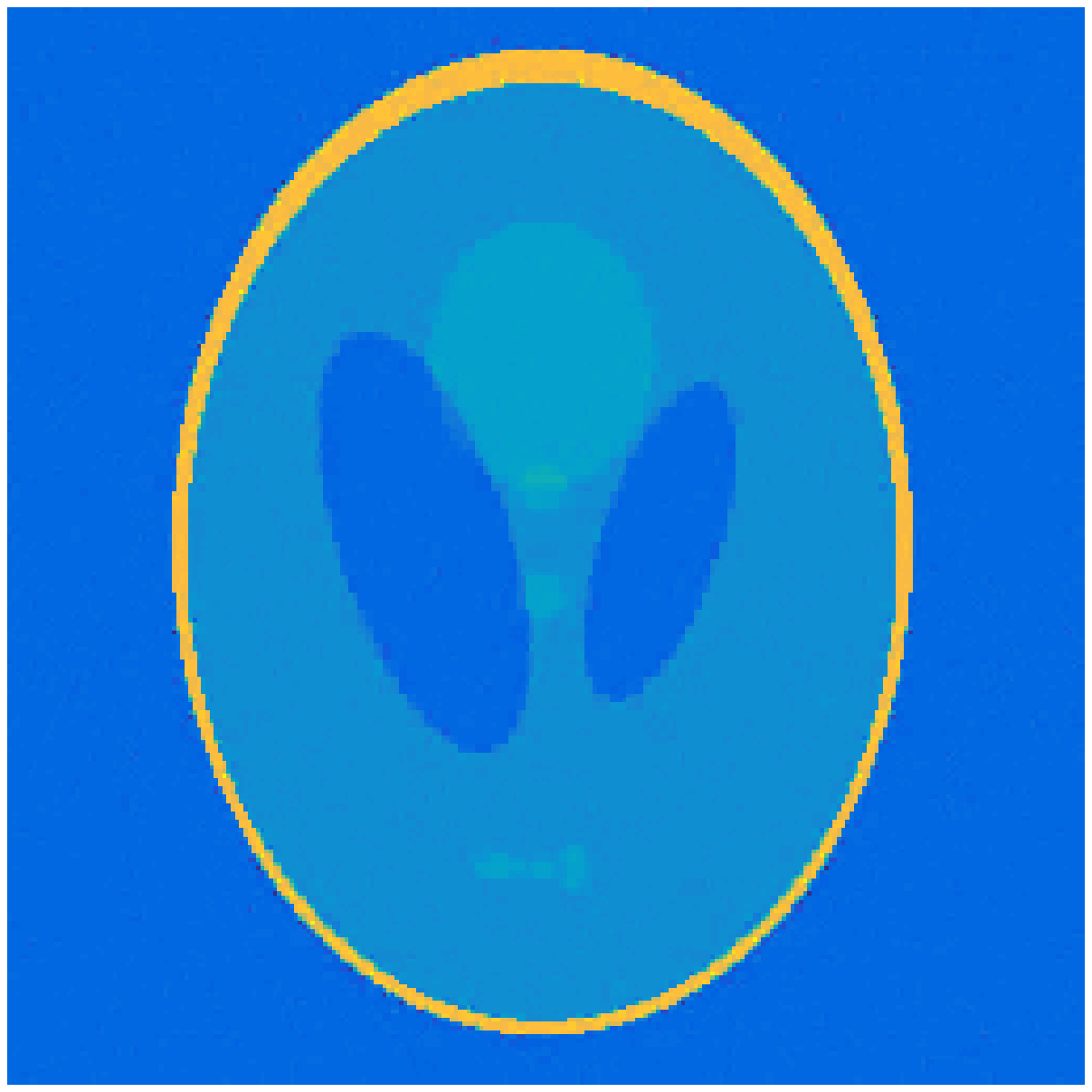}
\end{tabular}
\caption{Experiment 3: Best reconstructions computed by various solvers.}
\label{fig:exp3_sol}
\end{figure}
Figure \ref{fig:exp3_sol} shows the best reconstructions computed by each method considered in Figure \ref{fig:ex3_relerr_hybrid}. The best relative error and the iteration number (preceded by \#) are reported in brackets.

\section{Conclusions and Future Work}
\label{sec:conclusions}
In this paper, we described flexible hybrid iterative methods for computing approximate solutions to the (transformed) $\ell_p$-regularized problem, for $p\geq 1$.
To handle general (non-square) $\ell_p$-regularized least-squares problems, we introduced a flexible Golub-Kahan approach and exploited it within a Krylov-Tikhonov hybrid framework.
Theoretical results showed that the iterates correspond to solutions of a full-dimensional Tikhonov-like problem that has been projected onto a flexible Krylov subspace of increasing dimension.  We described various extensions for efficiently computing solutions that are sparse with respect to some invertible transformation.  Our proposed methods are \emph{efficient} in that they are matrix-free and avoid inner-outer schemes, and \emph{automatic} in that parameters such as regularization parameters and stopping iterations can be naturally selected within a hybrid framework.  Numerical results validate these observations.

Future work includes extensions to problems where $\bfPsi$ is not invertible, and also to nonlinear regularization functionals (e.g., total variation) and nonconvex problems.  Furthermore, by incorporating a multi-level decomposition, these flexible hybrid methods can be exploited in a multi-parameter regularization framework, where a different sparsity regularization parameter is incorporated for each level.

\section*{Acknowledgments}
This work was partially supported by NSF DMS 1654175 and NSF DMS 1723005 (J.~Chung). The authors would like to thank the Isaac Newton Institute for Mathematical Sciences for support and hospitality during the programme ``Variational methods and effective algorithms for imaging and vision'' when work on this paper was undertaken. This work was supported by: EPSRC grant numbers EP/K032208/1 and EP/R014604/1

\bibliographystyle{siamplain}
\bibliography{flexGMRES}

\end{document}